\documentclass[12pt,oneside,a4papre]{amsart}

\usepackage{amssymb}
\usepackage{amsmath}
\usepackage{amsthm}
\usepackage{amscd}

\usepackage{longtable}
\usepackage{mathrsfs}

\usepackage[dvipdfmx]{xcolor}
\usepackage[dvipdfmx]{pict2e}
\usepackage[dvipdfmx]{graphicx}
\usepackage{comment}
\usepackage{todonotes}

\usepackage{tikz}

\usepackage[all]{xy}

\usetikzlibrary{cd}

\theoremstyle{plain}
\newtheorem{theorem}{Theorem}[section]
\newtheorem{lemma}[theorem]{Lemma}
\newtheorem{corollary}[theorem]{Corollary}
\newtheorem{proposition}[theorem]{Proposition}
\newtheorem{remark}[theorem]{Remark}

\theoremstyle{definition}
\newtheorem{definition}[theorem]{Definition}
\newtheorem{example}[theorem]{Example}

\newcommand{\kah}{K\"{a}hler }
\newcommand{\idd}{i\partial\overline{\partial}}


\begin{document}
\title
[$L^2$-type Dolbeault isomorphisms and vanishing theorems] 
{$L^2$-type Dolbeault isomorphisms and vanishing theorems for logarithmic sheaves twisted by multiplier ideal sheaves}
\author{Yuta Watanabe}
\date{}

\begin{abstract}
   In this article, we first establish an $L^2$-type Dolbeault isomorphism for the sheaf of logarithmic differential forms twisted by the multiplier ideal sheaf. 
   By using this isomorphism and $L^2$-estimates equipped with a singular Hermitian metric, we obtain logarithmic vanishing theorems involving multiplier ideal sheaves on compact K\"ahler manifolds with simple normal crossing divisors.
\end{abstract}

\vspace{-5mm}

\maketitle

\tableofcontents

\vspace{-8mm}

\section{Introduction}

The vanishing theorem of cohomology is one of the central topics in several complex variables and complex algebraic geometry.
Starting with the Kodaira vanishing theorem (cf. \cite{Kod53}) for positive line bundles, many vanishing theorems have been studied, such as Nakano (cf. \cite{Nak73}, \cite{Ina22}) on vector bundles, Kawamata-Viehweg (cf. \cite{Kaw82}, \cite{Vie82}) for nef and big divisors, Bogomolove-Sommese (cf. \cite{Bog78}, \cite{Wu20}, \cite{Wat22a}, \cite{LMNWZ22}) on positivity for Kodaira dimension etc., and Nadel-Demailly (cf. \cite{Dem93}, \cite{Nad89}, \cite{Wat22b}) involving multiplier ideal sheaves.
Recently, Huang, Liu, Wan and Yang proved the following logarithmic vanishing theorem in \cite{HLWY16}: 

\begin{theorem}$(\mathrm{cf.~[HLWY22,\,Theorem\,1.1]})$\label{Log V-thm in HLWY16}
Let $X$ be a compact \kah manifold of dimension $n$ and $D=\sum^s_{j=1}D_j$ be a simple normal crossing divisor in $X$. 
Let $A$ be a holomorphic line bundle and $\Delta=\sum^s_{j=1}a_jD_j$ be an $\mathbb{R}$-divisor with $a_j\in[0,1]$ such that $A\otimes\mathcal{O}_X(\Delta)$ is a $k$-positive $\mathbb{R}$-line bundle in the sense of Definition \ref{def of k-positivity}.
Then for any nef line bundle $N$, we have 
\begin{align*}
    H^q(X,\Omega^p_X(\log D)\otimes A\otimes N)=0 \quad \mathit{for~any}~\,\,p+q\geq n+k.
\end{align*}
\end{theorem}

Let $\varphi$ be a plurisubharmonic function on $X$. 
We define the $\it{multiplier}$ $\it{ideal}$ $\it{sheaf}$ to be the ideal subsheaf $\mathscr{I}(\varphi)\subset\mathcal{O}_X$ of germs of holomorphic functions $f\in\mathcal{O}_x$ such that $|f|^2e^{-\varphi}$ is locally integrable near $x\in X$.
For a singular Hermitian metric $h$ on a holomorphic line bundle, the multiplier ideal sheaf is also defined by $\mathscr{I}(h):=\mathscr{I}(\varphi)$ where $h=e^{-\varphi}$ locally. 
It is well known that $\mathscr{I}(\varphi)$ is coherent (see \cite{Nad89}).

Our main purpose in this article is to establish logarithmic vanishing theorems involving multiplier ideal sheaves such as Nadel-Demailly.
Let $X$ be a compact \kah manifold and $Y=X\setminus D$ where $D$ is a simple normal crossing divisor. Let $F$ be a holomorphic vector bundle and $L$ be a holomorphic line bundle equipped with a singular Hermitian metric $h$ which is pseudo-effective. 
First, we obtain an $L^2$-type Dolbeault isomorphism for the logarithmic sheaf $\Omega^p_X(\log D)$ twisted by the multiplier ideal sheaf $\mathscr{I}(h)$ that follows from the $L^2$ fine resolution $(\mathscr{L}^{p,\ast}_{F\otimes L,h^F_Y\otimes h,\omega_P},\overline{\partial})$:
\begin{align*}
    0\longrightarrow \Omega_X^p(\log D)\otimes \mathcal{O}_X(F\otimes L)\otimes \mathscr{I}(h) \longrightarrow \mathscr{L}^{p,\ast}_{F\otimes L,h^F_Y\otimes h,\omega_P},
\end{align*}
where $h^F_Y$ is a suitably chosen smooth Hermitian metric on $F|_Y$ and $\omega_P$ is a Poincar\'e type metric on $Y$.
Similarly for a singular Hermitian metric $h$ on a holomorphic vector bundle $E$, if $h$ is Griffiths semi-positive then we get an $L^2$-type Dolbeault isomorphism that follows from the $L^2$ fine resolution $(\mathscr{L}^{p,\ast}_{F\otimes E\otimes\mathrm{det}\,E,h^F_Y\otimes h\otimes\mathrm{det}\,h,\omega_P},\overline{\partial})$:
\begin{align*}
    0\longrightarrow \Omega_X^p(\log D)\otimes \mathcal{O}_X(F)\otimes \mathscr{E}(h\otimes\mathrm{det}\,h) \longrightarrow \mathscr{L}^{p,\ast}_{F\otimes E\otimes\mathrm{det}\,E,h^F_Y\otimes h\otimes\mathrm{det}\,h,\omega_P},
\end{align*}
where $\mathscr{E}(h\otimes\mathrm{det}\,h)\subseteq\mathcal{O}_X(E\otimes\mathrm{det}\,E)$ is coherent from Griffiths semi-positivity of $h$ (see \cite{HI20}, \cite{Ina22}).
Then, we establish the following logarithmic vanishing using this isomorphism and $L^2$-estimates equipped with a singular Hermitian metric.


\begin{theorem}\label{Log V-thm of k-posi + psef no nu condition}
    Let $X$ be a compact \kah manifold and $D$ be a simple normal crossing divisor in $X$. 
    Let $A$ be a holomorphic line bundle and $L$ be a holomorphic line bundle equipped with a singular Hermitian metric $h$ which is pseudo-effective, i.e. $i\Theta_{L,h}\geq0$ in the sense of currents. 
    If $A\otimes\mathcal{O}_X(D)$ is $k$-positive then for any nef line bundle $N$, we have that 
    \begin{align*}
        H^q(X,K_X\otimes\mathcal{O}_X(D)\otimes A\otimes N\otimes L\otimes \mathscr{I}(h))&=0,\\
        H^n(X,\Omega_X^p(\log D)\otimes A\otimes N\otimes L\otimes \mathscr{I}(h))&=0
    \end{align*}
    for $p,q\geq k$.
\end{theorem}

Here, $\Omega^n_X(\log D)=K_X\otimes\mathcal{O}_X(D)$. 
This cohomology vanishing for $(n,q)$-forms follows immediately from already known results (see \cite{Hua20}, \cite{Wat22b}) and this paper gives an another proof using the $L^2$-type Dolbeault isomorphism for logarithmic sheaves.

By using the Lelong number as an invariant of singularities with respect to singular Hermitian metrics, we prove the following logarithmic vanishing theorem with certain degrees of freedom in positivity, analogous to Theorem \ref{Log V-thm in HLWY16}.
The $\it{Lelong}$ $\it{number}$ of a plurisubharmonic function $\varphi$ on $X$ is defined by 
\begin{align*}
    \nu(\varphi,x):=\liminf_{z\to x}\frac{\varphi(z)}{\log|z-x|}
\end{align*}
for some coordinate $(z_1,\cdots,z_n)$ around $x\in X$.
For the relationship between the Lelong number of $\varphi$ and the integrability of $e^{-\varphi}$, the following important result obtained by Skoda in \cite{Sko72} is known; If $\nu(\varphi,x)<1$ then $e^{-2\varphi}$ is integrable around $x$.
From this, particularly if $\nu(-\log h,x)<2$ 
then $\mathscr{I}(h)=\mathcal{O}_{X,x}$ immediately.

\begin{theorem}\label{Log V-thm of k-posi + psef}
    Let $X$ be a compact \kah manifold and $D=\sum^s_{j=1}D_j$ be a simple normal crossing divisor in $X$. 
    Let $A$ be a holomorphic line bundle and $L$ be a holomorphic line bundle equipped with a singular Hermitian metric $h$ which is pseudo-effective, i.e. $i\Theta_{L,h}\geq0$ in the sense of currents. 
    If there exists $0<\delta\leq1$ such that $\nu(-\log h,x)<2\delta$ for all points in $D$ and that for an $\mathbb{R}$-divisor $\Delta=\sum^s_{j=1}a_jD_j$ with $a_j\in(\delta,1]$, the $\mathbb{R}$-line bundle $A\otimes\mathcal{O}_X(\Delta)$ is $k$-positive. 
    Then for any nef line bundle $N$, we have that 
    \begin{align*}
        H^q(X,K_X\otimes\mathcal{O}_X(D)\otimes A\otimes N\otimes L\otimes \mathscr{I}(h))&=0,\\
        H^n(X,\Omega_X^p(\log D)\otimes A\otimes N\otimes L\otimes \mathscr{I}(h))&=0
    \end{align*}
    for $p,q\geq k$.
\end{theorem}

When using the above $L^2$-type Dolbeault isomorphism for $\Omega_X^p(\log D)\otimes \mathcal{O}_X(F\otimes L)\otimes \mathscr{I}(h)$ in this proof, 
from the Lelong number condition, there exists a degree of freedom for the suitably chosen smooth Hermitian metric $h^F_Y$ on $F|_Y$ in the $L^2$ fine resolution $(\mathscr{L}^{p,\ast}_{F\otimes L,h^F_Y\otimes h,\omega_P},\overline{\partial})$, which directly becomes a degree of freedom for positivity.

Furthermore, we obtain analogous results to Theorem \ref{Log V-thm of k-posi + psef no nu condition} and \ref{Log V-thm of k-posi + psef} for big line bundles in $\S 4.1$.
And we establish logarithmic vanishing theorems for singular Hermitian metrics on holomorphic vector bundles with Griffiths positivity in $\S 4.2$.
Here, one of them is as follows. Finally, a counterexample for the extension to Kodaira-Akizuki-Nakano type is given in $\S 4.3$.


\begin{theorem}\label{Log V-thm of k-posi + Grif no nu condition}
    Let $X$ be a projective manifold and $D$ be a simple normal crossing divisor in $X$\!. \!Let $A$ be a holomorphic line bundle and 
    $E$ be a holomorphic vector bundle equipped with a singular Hermitian metric $h$. \!We assume that $h$ is Griffiths semi-positive and that $A\otimes\mathcal{O}_X(D)$ is $k$-positive.
    Then for any nef line bundle $N$, we have that
    \begin{align*}
        H^q(X,K_X\otimes\mathcal{O}_X(D)\otimes A\otimes N\otimes \mathscr{E}(h\otimes\mathrm{det}\,h))&=0,\\
        H^n(X,\Omega^p_X(\log D)\otimes A\otimes N\otimes \mathscr{E}(h\otimes\mathrm{det}\,h))&=0
    \end{align*}
    for $p,q\geq k$.
\end{theorem}

\section{Preliminaries}

\subsection{Poincar\'e type metric and $\mathbb{R}$-divisors}

A divisor $D=\sum^s_{j=1}D_j$ on $X$ is called a $\it{simple}$ $\it{normal}$ $\it{crossing}$ $\it{divisor}$ if every irreducible component $D_j$ is smooth and all intersections are transverse.
The $\it{logarithmic}$ $\it{sheaf}$ $\Omega_X^p(\log D)$, introduced by Deligne in \cite{Del69}, is the sheaf of germs of differential $p$-forms on $X$ with at most logarithmic poles along $D$ whose sections on an open subset $U$ of $X$ are 
\begin{align*}
    \Gamma(U,\Omega^p_X(\log D)):=\{\alpha\in\Gamma(U,\Omega^p_X\otimes\mathcal{O}_X(D))~\mathrm{and}~d\alpha\in\Gamma(U,\Omega^{p+1}_X\otimes\mathcal{O}_X(D))\}.
\end{align*}

Define the complement $Y:=X\setminus D$. We can choose a local coordinate chart $(W;z_1,\ldots,z_n)$ of $X$ such that the locus of $D$ is given by $z_1\cdots z_t=0$ and $Y\cap W=W^*_r=(\Delta^*_r)^t\times(\Delta_r)^{n-t}$ where $\Delta_r$ (resp. $\Delta^*_r$) is the (resp. punctured) open disk of radius $r$ in the complex plane.
We give a \kah metric $\omega_Y$ only on the open manifold $Y$, which satisfies some special asymptotic behaviors along $D$.

\begin{definition}(Poincar\'e~type~metric)
    We say that the metric $\omega_Y$ on $Y$ is of $\mathit{Poincar\!\acute{e}}$ $\it{type}$ along $D$, if for each local coordinate chart $(W;z_1,\ldots,z_n)$ along $D$ the restriction $\omega_Y|_{W^*_r}$ is equivalent to the usual Poincar\'e type metric $\omega_P$ defined by 
    \begin{align*}
        \omega_P=i\sum^t_{j=1}\frac{dz_j\wedge d\overline{z}_j}{|z_j|^2(\log|z_j|^2)^2}+i\sum^n_{j=t+1}dz_j\wedge d\overline{z}_j.
    \end{align*}
\end{definition}

It is well-known that there always exists a \kah metric $\omega_Y$ on $Y$ which is of Poincar\'e type along $D$ (see [Zuc79,\,Section\,3]). Furthermore, this metric is complete and of finite volume.
The following example is used for the integrability of holomorphic sections with respect to the Poincar\'e type metrics.

\begin{example}\label{Example of integral}
    For any positive number $m$, the integral
    \begin{align*}
        \int^1_0r^\alpha(-\log r)^mdr 
    \end{align*}
    is finite if and only if $\alpha>-1$. In particular, if $m$ is a positive integer then we have
    \begin{align*}
        \int^1_0r^\alpha(-\log r)^mdr=\frac{m!}{(1+\alpha)^{m+1}}.
    \end{align*}
\end{example}

Second, we introduce the notion of $\mathbb{R}$-divisors and $\mathbb{R}$-line bundles.
\begin{itemize}
    \item [($a$)] $D$ is called an $\mathbb{R}$-$\it{divisor}$, if it is an element of $\mathrm{Div}_{\mathbb{R}}(X):=\mathrm{Div}(X)\otimes_{\mathbb{Z}}\mathbb{R}$. 
    In other words, an $\mathbb{R}$-divisor $D$ can be written as a finite sum of divisors with real coefficients, i.e. $D=\sum^k_{j=1} a_jD_j$ where $a_j\in\mathbb{R}$ and $D_j\in \mathrm{Div}(X)$.
    \item [($b$)] An $\mathbb{R}$-$\it{line~bundle}$ $L=\sum^k_{j=1}a_jL_j$ is a finite sum with some real numbers $a_1,\cdots,a_k$ and holomorphic line bundles $L_1,\cdots,L_k$.
\end{itemize}
For example, the $\mathbb{R}$-divisor $D=\sum^k_{j=1} a_jD_j\in\mathrm{Div}_{\mathbb{R}}(X)$ determines an $\mathbb{R}$-line bundle $\mathcal{O}_X(D)=\sum^k_{j=1} a_j\mathcal{O}_X(D_j)$.

\subsection{Singular Hermitian metrics and positivity}

First, we define positivity for smooth Hermitian metrics on holomorphic line bundles.

\begin{definition}\label{def of k-positivity}
    Let $X$ be a complex manifold of dimension $n$. 
    \begin{itemize}
        \item [($a$)] Let $L=\sum^k_{j=1}a_jL_j$ be a $\mathbb{R}$-line bundle over $X$, where $a_1,\cdots,a_k\in\mathbb{R}$ and $L_1,\cdots,L_k$ are holomorphic line bundles.
        We say that $L$ is $k$-$\it{positive}$ if there exist smooth Hermitian metrics $h_1,\cdots,h_k$ on $L_1,\cdots,L_k$ such that the curvature of the induced metric on $L$, which is explicity given by
        \begin{align*}
            i\Theta_{L,h}=i\sum^k_{j=1}a_j\Theta_{L_j,h_j},
        \end{align*}
        is $k$-positive, i.e. $i\Theta_{L,h}$ is semi-positive and has at least $n-k+1$ positive eigenvalues at every point of $X$.
        \item [($b$)] An $\mathbb{R}$-divisor $D=\sum^k_{j=1}a_jD_j$ is said to be $k$-$\it{positive}$ if the induced $\mathbb{R}$-line bundle $\mathcal{O}_X(D)=\sum^k_{j=1}a_j\mathcal{O}_X(D_j)$ is $k$-positive.
    \end{itemize}
\end{definition}

We introduce the definition of singular Hermitian metrics on line bundles and its positivity. 

\begin{definition}$(\mathrm{cf.~[Dem93],\,[Dem10,\,Chapter\,3]})$
    A $\it{singular~Hermitian~metric}$ $h$ on a line bundle $L$ is a metric which is given in any trivialization $\tau:L|_U\xrightarrow{\simeq} U\times\mathbb{C}$ by 
    \begin{align*}
        ||\xi||_h=|\tau(\xi)|e^{-\varphi}, \qquad x\in U,\,\,\xi\in L_x
    \end{align*}
    where $\varphi\in\mathcal{L}^1_{loc}(U)$ is an arbitrary function, called the weight of the metric with respect to the trivialization $\tau$.
\end{definition}

\begin{definition}\label{def of psef & big}$(\mathrm{cf.~[Wat22b,\,Definition\,3.2]})$
    Let $L$ be a holomorphic line bundle on a complex manifold $X$ equipped with a singular Hermitian metric $h$. 
    \begin{itemize}
        \item [($a$)] $h$ is $\it{singular}$ $\it{semi}$-$\it{positive}$ if $i\Theta_{L,h}\geq0$ in the sense of currents,
        i.e. the weight of $h$ with respect to any trivialization coincides with some plurisubharmonic function almost everywhere.
        \item [($b$)] $h$ is $\it{singular}$ $\it{positive}$ if the weight of $h$ with respect to any trivialization coincides with some strictly plurisubharmonic function almost everywhere.
    \end{itemize}
\end{definition}

Note that, singular semi-positivity is coincides with pseudo-effective on compact complex manifolds. 
Furthermore, singular positivity also coincide with big on compact \kah manifolds by Demailly's definition and characterization (see \cite{Dem93}, [Dem10,\,Chapter\,6]). 

We define singular Hermitian metrics on vector bundles and its positivity such as Griffiths and (dual) Nakano.

\begin{definition}$(\mathrm{cf.~[BP08,~Section~3],~[Rau15,~Definition~1.1]~and~[PT18,~Definition}$ $\mathrm{2.2.1]})$ 
    We say that $h$ is a $\it{singular~Hermitian~metric}$ on $E$ if $h$ is a measurable map from the base manifold $X$ to the space of non-negative Hermitian forms on the fibers satisfying $0<\mathrm{det}\,h<+\infty$ almost everywhere.
\end{definition}

\begin{definition}\label{def Griffiths semi-posi sing}$(\mathrm{cf.~[BP08,~Definition~3.1],~[Rau15,~Definition~1.2]~and~[PT18,~Def}$- 
    $\mathrm{inition~2.2.2]})$
    We say that a singular Hermitian metric $h$ is 
    \begin{itemize}
        \item [(1)] $\textit{Griffiths semi-negative}$ if $||u||_h$ is plurisubharmonic for any local holomorphic section $u\in\mathcal{O}(E)$ of $E$.
        \item [(2)] $\textit{Griffiths semi-positive}$ if the dual metric $h^*$ on $E^*$ is Griffiths semi-negative.
    \end{itemize}
\end{definition}

Let $h$ be a smooth Hermitian metric on $E$ and $u=(u_1,\cdots,u_n)$ be an $n$-tuple of holomorphic sections of $E$. We define $T^h_u$, an $(n-1,n-1)$-form through
\begin{align*}
    T^h_u=\sum^n_{j,k=1}(u_j,u_k)_h\widehat{dz_j\wedge d\overline{z}_k}
\end{align*}
where $(z_1,\cdots,z_n)$ are local coordinates on $X$, and $\widehat{dz_j\wedge d\overline{z}_k}$ denotes the wedge product of all $dz_i$ and $d\overline{z}_i$ expect $dz_j$ and $d\overline{z}_k$, 
multiplied by a constant of absolute value $1$, chosen so that $T_u^h$ is a positive form. Then a short computation yields that $(E,h)$ is Nakano semi-negative if and only if $T^h_u$ is plurisubharmonic in the sense that $\idd T^h_u\geq0$ (see \cite{Ber09},\,\cite{Rau15}).
In the case of $u_j=u_k=u$, $(E,h)$ is Griffiths semi-negative.

Let $h$ be a singular Hermitian metric of $E$. For any $n$-tuple of local holomorphic sections $u=(u_1,\cdots,u_n)$, we say that the $(n-1,n-1)$-form $T^h_u$ is plurisubharmonic if $\idd T^h_u\geq0$ in the sense of currents.


\begin{definition}\label{def Nakano semi-negative as Raufi}$\mathrm{(cf.\,[Rau15,\,Section\,1]})$
    We say that a singular Hermitian metric $h$ on $E$ is $\it{Nakano}$ $\it{semi}$-$\it{negative}$ if the $(n-1,n-1)$-form $T^h_u$ is plurisubharmonic for any $n$-tuple of holomorphic sections $u=(u_1,\cdots,u_n)$.
\end{definition}

\begin{definition}\label{def dual Nakano semi-posi sing}$\mathrm{(cf.\,[Wat22a,\,Definition\,4.5]})$
    We say that a singular Hermitian metric $h$ on $E$ is $\it{dual}$ $\it{Nakano}$ $\it{semi}$-$\it{positive}$ if the dual metric $h^*$ on $E^*$ is Nakano semi-negative.
\end{definition}

We already know the following definition of Nakano semi-positivity using $L^2$-estimates in \cite{Wat22b} as follows, which is based on the $L^2$-estimate condition (see \cite{HI20}, \cite{DNWZ20}, \cite{Ina22}, \cite{Wat22b}) and is equivalent to the usual definition for the smooth case.

\begin{definition}\label{def Nakano semi-posi sing}$(\mathrm{cf.~[Wat22b,\,Definition\,3.9]})$
    Assume that $h$ is a Griffiths semi-positive singular Hermitian metric. We say that $h$ is 
    $L^2$-$\it{type}$ $\it{Nakano}$ $\it{semi}$-$\it{positive}$ if for any positive integer $k\in\{1,\cdots,n\}$, any Stein coordinate $S$, any \kah metric $\omega_S$ on $S$ and
    any smooth Hermitian metric $h_F$ on any holomorphic vector bundle $F$ such that $A^{n,s}_{F,h_F,\omega_S}:=[i\Theta_{F,h_F},\Lambda_{\omega_S}]>0$ on $\Lambda^{n,s}T^*_S\otimes F$ for $s\geq k$,
    we have that any positive integer $q\geq k$ and any $\overline{\partial}$-closed $f\in L^2_{n,q}(S,E\otimes F,h\otimes h_F,\omega_S)$ 
    there exists $u\in L^2_{n,q-1}(S,E\otimes F,h\otimes h_F,\omega_S)$ satisfying $\overline{\partial}u=f$ and 
    \begin{align*}
        \int_S|u|^2_{h\otimes h_F,\omega_S}dV_{\omega_S}\leq\int_S\langle B^{-1}_{h_F,\omega_S}f,f\rangle_{h\otimes h_F,\omega_S}dV_{\omega_S},
    \end{align*}
    where $B_{h_F,\omega_S}=[i\Theta_{F,h_F}\otimes\mathrm{id}_E,\Lambda_{\omega_S}]$. Here we assume that the right-hand side is finite.
\end{definition}

Here, for singular Hermitian metrics we cannot always define the curvature currents with measure coefficients (see \cite{Rau15}). 
The above definitions \ref{def Griffiths semi-posi sing}-\ref{def Nakano semi-posi sing} dose not require the use of curvature currents.

For singular Hermitian metrics $h$ on holomorphic vector bundles $E$, we introduce the $L^2$-subsheaf $\mathscr{E}(h)$ of $\mathcal{O}(E)$ with respect to $h$ that is analogous to the multiplier ideal sheaf.
In fact, $\mathscr{E}(h)=\mathcal{O}(E)\otimes\mathscr{I}(h)$ if $E$ is a holomorphic line bundle.

\begin{definition}\label{def of L2-subsheaf}$(\mathrm{cf.\,[deC98,\,Definition\,2.3.1]})$
    Let $h$ be a singular Hermitian metric on $E$. We define the $L^2$-subsheaf $\mathscr{E}(h)$ of germs of local holomorphic sections of $E$ as follows:
    \begin{align*}
        \mathscr{E}(h)_x:=\{s_x\in\mathcal{O}(E)_x\mid|s_x|^2_h~ \mathrm{is ~locally ~integrable ~around~} x\}.
    \end{align*}
\end{definition}

In \cite{Nad89}, Nadel proved that $\mathscr{I}(h)$ is coherent by using the H\"ormander $L^2$-estimate. 
After that, Hosono and Inayama proved that $\mathscr{E}(h)$ is coherent if $h$ has Nakano semi-positivity
and that $\mathscr{E}(h\otimes\mathrm{det}\,h)$ is coherent if $h$ is Griffiths semi-positive in \cite{HI20} and \cite{Ina22}.
Finally we introduce the following definition of strictly positivity for Griffiths and (dual) Nakano. 

\begin{definition}\label{def sing strictly positive of Grif and (dual) Nakano}$\mathrm{(cf.\,[Ina22,\,Definition\,2.16],\,[Wat22a,\,Definition\,4.11])}$ 
    Let $(X,\omega)$ be a \kah manifold and $h$ be a singular Hermitian metric on $E$.
    \begin{itemize}
        \item We say that $h$ is $\textit{strictly Griffiths}$ $\delta_\omega$-$\textit{positive}$ if for any open subset $U$ and any \kah potential $\varphi$ of $\omega$ on $U$, $he^{\delta\varphi}$ is Griffiths semi-positive on $U$.
        \item We say that $h$ is $L^2$-$\it{type}$ $\textit{strictly Nakano}$ $\delta_\omega$-$\textit{positive}$ if for any open subset $U$ and any \kah potential $\varphi$ of $\omega$ on $U$, $he^{\delta\varphi}$ is $L^2$-type Nakano semi-positive on $U$ in the sense of Definition \ref{def Nakano semi-posi sing}.
        \item We say that $h$ is $\textit{strictly dual Nakano}$ $\delta_\omega$-$\textit{positive}$ if for any open subset $U$ and any \kah potential $\varphi$ of $\omega$ on $U$, $he^{\delta\varphi}$ is dual Nakano semi-positive on $U$.
    \end{itemize}
\end{definition}

The relationship is already known for strictly positivity of Griffiths and Nakano. 
By using [Wat22b,\,Theorem\,1.5], we can immediately obtain the following relationship for dual Nakano as well.

\begin{theorem}\label{h s-Gri then h * det h s-Nak}$\mathrm{(cf.\,[Ina22,\,Theorem\,3.6])}$
    Let $\omega$ be a \kah metric on a \kah manifold $X$ and $E$ be a holomorphic vector bundle of $\mathrm{rank}\,r$ over $X$ equipped with a singular Hermitian metric $h$. 
    If $h$ is strictly Griffiths $\delta_\omega$-positive then the singular Hermitian metric $h\otimes\mathrm{det}\,h$ on $E\otimes\mathrm{det}\,E$ is $L^2$-type strictly Nakano $(r+1)\delta_\omega$-positive.
\end{theorem}

\begin{theorem}\label{h s-Gri then h * det h s-dual Nak}
    Let $\omega$ be a \kah metric on a \kah manifold $X$ and $E$ be a holomorphic vector bundle of $\mathrm{rank}\,r$ over $X$ equipped with a singular Hermitian metric $h$. 
    If $h$ is strictly Griffiths $\delta_\omega$-positive then the singular Hermitian metric $h\otimes\mathrm{det}\,h$ on $E\otimes\mathrm{det}\,E$ is strictly dual Nakano $(r+1)\delta_\omega$-positive.
\end{theorem}

\subsection{$L^2$-estimates and $L^2$-type Dolbeault complexes}

Here, a function $\psi:X\to[-\infty,+\infty)$ is said to be $\it{exhaustive}$ if all sublevel sets $X_c:=\{z\in X\mid \psi(x)<c\},\,c<\sup_X\psi$, are relatively compact.
A complex manifold is said to be $\it{weakly}$ $\it{pseudoconvex}$ if there exists a smooth exhaustive plurisubharmonic function.

The following $L^2$-estimates equipped with a singular Hermitian metric is used to prove logarithmic vanishing theorems.

\begin{theorem}\label{L2-estimate of k-posi + psef}$(\mathrm{cf.~[Wat22b,\,Corollary\,4.4]})$
    Let $X$ be a compact \kah manifold, $D$ be a simple normal crossing divisor on $X$ and $Y:=X\setminus D$.
    Let $\omega_Y$ be a \kah metric on $Y$, $A$ be a holomorphic line bundle and $L$ be a holomorphic line bundle equipped with a singular semi-positive Hermitian metric $h$, i.e. $i\Theta_{L,h}\geq0$ in the sense of currents.
    If there exists a smooth Hermitian metric $h_A$ on $A$ over $Y$ such that the curvature operator $A^{p,q}_{A,h_A,\omega_Y}:=[i\Theta_{A,h_A},\Lambda_{\omega_Y}]$ on $\Lambda^{p,q}T^*_X\otimes A$ is positive definite for $p+q\geq n+k$, for example $h_A$ is $k$-positive, then we have the following
    \begin{itemize}
        \item [($a$)] For any $q\geq k$ and any $f\in L^2_{n,q}(Y,A\otimes L,h_A\otimes h,\omega_Y)$ satisfying $\overline{\partial}f=0$ and 
        $\int_Y\langle B^{-1}_{h_A,\omega_Y}f,f\rangle_{h_A\otimes h,\omega_Y}dV_{\omega_Y}<+\infty$,
        there exists $u\in L^2_{n,q-1}(Y,A\otimes L,h_A\otimes h,\omega_Y)$ such that $\overline{\partial}u=f$ and 
        \begin{align*}
            \int_Y|u|^2_{h_A\otimes h,\omega_Y}dV_{\omega_Y}\leq\int_Y\langle B^{-1}_{h_A,\omega_Y}f,f\rangle_{h_A\otimes h,\omega_Y}dV_{\omega_Y},
        \end{align*}
        \item [($b$)] For any $p\geq k$ and any $f\in L^2_{p,n}(Y,A\otimes L,h_A\otimes h,\omega_Y)$ satisfying $\overline{\partial}f=0$ and 
        $\int_Y\langle B^{-1}_{h_A,\omega_Y}f,f\rangle_{h_A\otimes h,\omega_Y}dV_{\omega_Y}<+\infty$,
        there exists $u\in L^2_{p,n-1}(Y,A\otimes L,h_A\otimes h,\omega_Y)$ such that $\overline{\partial}u=f$ and 
        \begin{align*}
            \int_Y|u|^2_{h_A\otimes h,\omega_Y}dV_{\omega_Y}\leq\int_Y\langle B^{-1}_{h_A,\omega_Y}f,f\rangle_{h_A\otimes h,\omega_Y}dV_{\omega_Y},
        \end{align*}
    \end{itemize}
    where $B_{h_A,\omega_Y}=[i\Theta_{A,h_A}\otimes\mathrm{id}_L,\Lambda_{\omega_Y}]$.
\end{theorem}

\begin{remark}\label{Rem for L-estimate k-posi + psef}
    Theorem \ref{L2-estimate of k-posi + psef} is similarly shown by almost the same proof as $\mathrm{[Wat22b},$ $\mathrm{Corollary\,4.4]}$, although it is unclear whether $Y$ is weakly pseudoconvex. 
    Indeed, $Y$ has a (Poincar\'e type) complete \kah metric, and we can use Demailly's approximation to $h$ on $X$ due to the compactness of $X$.
\end{remark}

Here, $h\otimes\mathrm{det}\,h$ is $L^2$-type Naknao semi-positive and dual Nakano positive if $h$ is Griffiths semi-positive (see [Ina22,\,Theorem\,1.3], [Wat22b,\,Theorem\,1.5])

\begin{theorem}\label{L2-estimate of k-posi + Grif}$(\mathrm{cf.~[Wat22b,\,Theorem\,4.8\,and\,4.9]})$
    Let $X$ be a weakly pseudoconvex \kah manifold equipped with a effective divisor $D$ such that $\mathcal{O}_X(D)$ is positive line bundle and $\omega$ be a \kah metric.
    Let $(A,h_A)$ be a $k$-positive line bundle and $E$ be a holomorphic vector bundle equipped with a singular Hermitian metric $h$.
    We assume that $h$ is Griffiths semi-positive on $X$. Then we have the following
    \begin{itemize}
        \item [($a$)] For any $q\geq k$ and any $f\in L^2_{n,q}(X,A\otimes E\otimes\mathrm{det}\,E,h_A\otimes h\otimes\mathrm{det}\,h,\omega)$ satisfying $\overline{\partial}f=0$ and $\int_X\langle B^{-1}_{h_A,\omega}f,f\rangle_{h_A\otimes h\otimes\mathrm{det}\,h,\omega}dV_\omega<+\infty$,
        there exists $u\in L^2_{n,q-1}(X,A\otimes E\otimes\mathrm{det}\,E,h_A\otimes h\otimes\mathrm{det}\,h,\omega)$ satisfies $\overline{\partial}u=f$ and 
        \begin{align*}
            \int_X|u|^2_{h_A\otimes h\otimes\mathrm{det}\,h,\omega}dV_{\omega}\leq\int_X\langle B^{-1}_{h_A,\omega}f,f\rangle_{h_A\otimes h\otimes\mathrm{det}\,h,\omega}dV_\omega,
        \end{align*}
        \item [($a$)] For any $p\geq k$ and any $f\in L^2_{p,n}(X,A\otimes E\otimes\mathrm{det}\,E,h_A\otimes h\otimes\mathrm{det}\,h,\omega)$ satisfying $\overline{\partial}f=0$ and $\int_X\langle B^{-1}_{h_A,\omega}f,f\rangle_{h_A\otimes h\otimes\mathrm{det}\,h,\omega}dV_\omega<+\infty$,
        there exists $u\in L^2_{p,n-1}(X,A\otimes E\otimes\mathrm{det}\,E,h_A\otimes h\otimes\mathrm{det}\,h,\omega)$ satisfies $\overline{\partial}u=f$ and 
        \begin{align*}
            \int_X|u|^2_{h_A\otimes h\otimes\mathrm{det}\,h,\omega}dV_{\omega}\leq\int_X\langle B^{-1}_{h_A,\omega}f,f\rangle_{h_A\otimes h\otimes\mathrm{det}\,h,\omega}dV_\omega,
        \end{align*}
    \end{itemize}
    where $B_{h_A,\omega}=[i\Theta_{A,h_A}\otimes\mathrm{id}_{E\otimes\mathrm{det}\,E},\Lambda_\omega]$.
\end{theorem}

For singular Hermitian metrics $h$ on $E$, we define the subsheaf $\mathscr{L}^{p,q}_{E,h}$ of germs of $(p,q)$-forms $u$ with values in $E$ and with measurable coefficients such that both $|u|^2_{h}$ and $|\overline{\partial}u|^2_h$ are locally integrable.

For an $L^2$-type Dolbeault complex with respect to the subsheaf $\mathscr{L}^{p,\ast}_{E,h}$, we already know the following $L^2$-type Dolbeault isomorphisms involving $L^2$-subsheaves.

\begin{theorem}\label{L2-type Dolbeault isomorphism with sHm}$(\mathrm{cf.~[Wat22b,\,Theorem\,5.4]})$
    Let $X$ be a complex manifold of dimension $n$ and $(F,h_F)$ be a holomorphic vector bundle. 
    Let $L$ be a holomorphic line bundle equipped with a singular Hermitian metric $h_L$ and $E$ be a holomorphic vector bundle equipped with a singular Hermitian metric $h_E$.
    Then we have the following
    \begin{itemize}
        \item [($a$)] If $h_L$ is singular semi-positive, then we have an exact sequence of sheaves 
        \begin{align*}
            0\longrightarrow \Omega_X^p\otimes \mathcal{O}_X(F\otimes L)\otimes \mathscr{I}(h_L) \longrightarrow \mathscr{L}^{p,\ast}_{F\otimes L,h_F\otimes h_L}.
        \end{align*}
        \item [($b$)] If $h_E$ is $L^2$-type Nakano semi-positive, then we get an exact sequence of sheaves 
        \begin{align*}
            0\longrightarrow K_X\otimes \mathcal{O}_X(F)\otimes \mathscr{E}(h_E) \longrightarrow \mathscr{L}^{n,\ast}_{F\otimes E,h_F\otimes h_E}.
        \end{align*}
        \item [($c$)] If $h_E$ is Griffiths semi-positive, then we have an exact sequence of sheaves 
        \begin{align*}
            0\longrightarrow \Omega_X^p\otimes \mathcal{O}_X(F)\otimes \mathscr{E}(h_E\otimes\mathrm{det}\,h_E) \longrightarrow \mathscr{L}^{p,\ast}_{F\otimes E\otimes\mathrm{det}\,E,h_F\otimes h_E\otimes\mathrm{det}\,h_E}.
        \end{align*}
    \end{itemize}
    In particular, $L^2$-type Dolbeault isomorphisms are obtained from these. For example, $H^q(X,\Omega_X^p\otimes F\otimes L\otimes \mathscr{I}(h_L))\cong H^q(\Gamma(X,\mathscr{L}^{p,\ast}_{F\otimes L,h_F\otimes h_L}))$ in the case of $(a)$.
\end{theorem}

Let $D$ be a simple normal crossing divisor and $Y\!:=\!X\setminus \!D$.
Let $\omega_Y$ is a smooth \kah metric on $Y$ and $h^E_Y$ is a singular Hermitian metric on $E|_Y$. The sheaf $\mathscr{L}^{p,q}_{E,h_Y^E,\omega_Y}$ over $X$ is defined as follows.
For any open subset $U$ of $X$, the section space $\mathscr{L}^{p,q}_{E,h_Y^E,\omega_Y}(U)$ consists of $E$-valued $(p,q)$-forms $u$ with measurable coefficients 
such that the $L^2$ norms of both $u$ and $\overline{\partial}u$, i.e. $|u|^2_{h^E_Y,\omega_Y}$ and $|\overline{\partial}u|^2_{h^E_Y,\omega_Y}$, are integrable on $U\setminus D$. 

It is well-known that the $\overline{\partial}$-operator with respect to $\mathscr{L}^{p,q}_{E,h_Y^E,\omega_Y}$ form an $L^2$-type Dolbeault complex $(\mathscr{L}^{p,\ast}_{E,h_Y^E,\omega_Y},\overline{\partial})$:
\begin{align*}
    \xymatrix{
    0 \ar[r] & \mathrm{ker}\,\overline{\partial}_0 \hookrightarrow \mathscr{L}^{p,0}_{E,h_Y^E,\omega_Y} \ar[r]^-{\overline{\partial}_0} & \mathscr{L}^{p,1}_{E,h_Y^E,\omega_Y} \ar[r]^-{\overline{\partial}_1} & \cdots \ar[r]^-{\overline{\partial}_{n-1}} & \mathscr{L}^{p,n}_{E,h_Y^E,\omega_Y} \ar[r] & 0,
    }
\end{align*}
where $\overline{\partial}_j=\overline{\partial}_{E}=\overline{\partial}\otimes\mathrm{id}_{E}$.

Let $F$ be a holomorphic vector bundle over $X$. 
For any local coordinate chart $(W;z_1,\ldots,z_n)$ along $D$, two smooth Hermitian metrics $h_1$ and $h_2$ on $F$ defined on $Y$ are said to be equivalent along $D$ on $W$ 
if there is a positive constant $C$ such that $(1/C)h_1\leq h_2\leq Ch_1$ on $W\setminus D$.
In this case we write $h_1\sim h_2$ on $W$.

Let $E$ be a holomorphic vector bundle over $X$ equipped with a singular Hermitian metric $h$.
If $h_1\sim h_2$ on $W$ then we write $h_1\otimes h\sim h_2\otimes h$ on $W$.
Note that we get $\mathscr{L}^{p,q}_{F\otimes E,h_1\otimes h,\omega_Y}(W)=\mathscr{L}^{p,q}_{F\otimes E,h_2\otimes h,\omega_Y}(W)$ if $h_1\otimes h\sim h_2\otimes h$ on $W$.

\section{$L^2$-type Dolbeault isomorphisms for logarithmic sheaves} 

\subsection{An $L^2$-type Dolbeault isomorphism for logarithmic sheaves twisted by multiplier ideal sheaves}

In this subsection, we present the following $L^2$-type Dolbeault sheaf resolution.

\begin{theorem}\label{L2-type Log Dolbeault isomorphism with sHm line}
    Let $(X,\omega)$ be a compact \kah manifold, $D=\sum^s_{j=1}D_j$ be a simple normal crossing divisor in $X$ and $\omega_P$ be a smooth \kah metric on $Y:=X\setminus D$ which is of Poincar\'e type along $D$. 
    Let $(F,h_F)$ be a holomorphic vector bundle over $X$, $(L,h)$ be a pseudo-effective line bundle over $X$, i.e. $i\Theta_{L,h}\geq0$ in the sense of currents, and $\sigma_j$ be the defining section of $D_j$. Fix smooth Hermitian metrics $||\bullet||_{D_j}$ on $\mathcal{O}_X(D_j)$. 

    Then for any enough large integer $\alpha>0$ there exists a smooth Hermitian metric 
    \begin{align*}
        h^F_Y=h_F\prod^s_{j=1}||\sigma_j||^2_{D_j}(\log ||\sigma_j||^2_{D_j})^{2\alpha}
    \end{align*}
    on $F|_Y$ such that the sheaf $\Omega_X^p(\log D) \otimes\mathcal{O}_X(F\otimes L)\otimes \mathscr{I}(h) $ over $X$ enjoys a fine resolution 
    given by the $L^2$-Dolbeault complex $(\mathscr{L}^{p,\ast}_{F\otimes L,h^F_Y\otimes h,\omega_P},\overline{\partial})$.

    Moreover, if there exists $0<\delta\leq1$ such that $\nu(-\log h,x)<2\delta$ for all points in $D$ then for any enough large integer $\alpha>0$ and any $\tau_j\in(\delta,1]$ there exists a smooth Hermitian metric 
    \begin{align*}
        h^F_Y=h_F\prod^s_{j=1}||\sigma_j||^{2\tau_j}_{D_j}(\log ||\sigma_j||^2_{D_j})^{2\alpha} 
    \end{align*}
    on $F|_Y$ such that the sheaf $\Omega_X^p(\log D) \otimes\mathcal{O}_X(F\otimes L)\otimes \mathscr{I}(h) $ over $X$ enjoys a fine resolution 
    given by the $L^2$-Dolbeault complex $(\mathscr{L}^{p,\ast}_{F\otimes L,h^F_Y\otimes h,\omega_P},\overline{\partial})$.

    In other words, we have the $L^2$-type Dolbeault sheaf resolution:
    \begin{align*}
        0\longrightarrow \Omega_X^p(\log D)\otimes \mathcal{O}_X(F\otimes L)\otimes \mathscr{I}(h) \longrightarrow \mathscr{L}^{p,\ast}_{F\otimes L,h^F_Y\otimes h,\omega_P} \tag*{($\ast$)},
    \end{align*}
    where $\mathscr{L}^{p,q}_{F\otimes L,h^F_Y\otimes h,\omega_P}$ is a fine sheaf for any $0\leq p,q\leq n$. In particular, we get
    \begin{align*}
        H^q(X,\Omega_X^p(\log D)\otimes F\otimes L \otimes \mathscr{I}(h))\cong H^q(\Gamma(X,\mathscr{L}^{p,\ast}_{F\otimes L,h^F_Y\otimes h,\omega_P})). 
    \end{align*}
\end{theorem}

To prove the theorem, we first show the following proposition.

\begin{proposition}\label{L2-type Log Dol isom at q=0 with sHm line}
    Let $X,Y,D,\omega,\omega_P,F,h_F,L,h,\sigma_j$ and $||\bullet||_{D_j}$ be as in Theorem \ref{L2-type Log Dolbeault isomorphism with sHm line}.

    Then for any integer $\alpha>0$ there exists a smooth Hermitian metric 
    \begin{align*}
        h^F_Y:=h^F_{\alpha}=h_F\prod^s_{j=1}||\sigma_j||^2_{D_j}(\log ||\sigma_j||^2_{D_j})^{2\alpha}
    \end{align*}
    on $F|_Y$ such that the following $L^2$-Dolbeault complex $(\mathscr{L}^{p,\ast}_{F\otimes L,h^F_Y\otimes h,\omega_P},\overline{\partial})$ is exact at $q=0$
    \begin{align*}
        0\longrightarrow \Omega_X^p(\log D)\otimes \mathcal{O}_X(F\otimes L)\otimes \mathscr{I}(h) \longrightarrow \mathscr{L}^{p,\ast}_{F\otimes L,h^F_Y\otimes h,\omega_P}, 
    \end{align*}
    i.e. $\mathrm{ker}\,\overline{\partial}_0=\Omega_X^p(\log D)\otimes \mathcal{O}_X(F\otimes L)\otimes \mathscr{I}(h)$.

    Moreover, if there exists $0<\delta\leq1$ such that $\nu(-\log h,x)<2\delta$ for all points in $D$ then for any integer $\alpha>0$ and any $\tau_j\in(\delta,1]$ there exists a smooth Hermitian metric 
    \begin{align*}
        h^F_Y:=h^F_{\alpha,\tau}=h_F\prod^s_{j=1}||\sigma_j||^{2\tau_j}_{D_j}(\log ||\sigma_j||^2_{D_j})^{2\alpha} 
    \end{align*}
    on $F|_Y$ such that the $L^2$-Dolbeault complex $(\mathscr{L}^{p,\ast}_{F\otimes L,h^F_Y\otimes h,\omega_P},\overline{\partial})$ over $X$ is exact at $q=0$.
\end{proposition}

For any fixed point $x_0\in Y=X\setminus D$, there exists an open neighborhood $U$ of $x_0$ such that $U\subset Y$, i.e. $U\cap D=\emptyset$.
Then we get $\Omega_X^p(\log D)=\Omega_X^p$ on $U$. Since Theorem \ref{L2-type Dolbeault isomorphism with sHm}, we have the exactness of $(\ast)$ on $Y$. 
Thus, it is sufficient that Theorem \ref{L2-type Log Dolbeault isomorphism with sHm line} and Proposition \ref{L2-type Log Dol isom at q=0 with sHm line} show the exactness of $(\ast)$ for all points on $D$.

\begin{proof}
    For any fixed point $x_0\in D$, let $(W;z_1,\ldots,z_n)$ be a local coordinate chart centered at $x_0$ along $D$ such that the locus of $D$ is given by $z_1\cdots z_t=0$ and that $F$ is trivial, i.e. $F|_W=W\times\mathbb{C}^r:=\underline{\mathbb{C}^r}$.
    Here, $W=W_r:=\Delta^n_r$, $Y\cap W=W\setminus D:=W^*_r=(\Delta^*_r)^t\times(\Delta_r)^{n-t}$ and $r\in(0,1)$.
    By the assumption $i\Theta_{L,h}\geq0$, the function $-\log h$ is plurisubharmonic on $W$ and $\sup_W-\log h<+\infty$, i.e. $\inf_Wh>0$. 
    Let $I_F$ be a trivial Hermitian metric of $F=\underline{\mathbb{C}^r}$ on $W$ and $b=(b_1,\ldots,b_r)$ be a holomorphic frame of $\underline{\mathbb{C}^r}$ on $W$ where $b_j=(0,\cdots,0,1,0,\cdots,0)$ and $b$ is orthonormal with respect to $I_F$. 

    Note that the smooth Hermitian metric $h^F_{\alpha,\tau}|_{W^*_r}$ is equivalent to the following smooth Hermitian metric
    \begin{align*}
        I^F_{\alpha,\tau}=I_F\prod^t_{j=1}|z_j|^{2\tau_j}(\log|z_j|^2)^{2\alpha} \quad \mathrm{on}~\,W^*_r,
    \end{align*}
    where $h_F\sim I_F$ and that $\mathscr{L}^{p,q}_{F\otimes L,h^F_Y\otimes h,\omega_Y}(W_r)=\mathscr{L}^{p,q}_{\underline{\mathbb{C}^r}\otimes L,I^F_{\alpha,\tau}\otimes h,\omega_Y}(W_r)$.
    Denote 
    \begin{align*}
        e_\lambda=e_L\otimes b_\lambda, \quad \zeta_j=\frac{1}{z_j}dz_j ~ \mathrm{for} ~ 1\leq j\leq t \quad \mathrm{and} \quad \zeta_j=dz_j ~ \mathrm{for} ~ t+1\leq j\leq n.  
    \end{align*}

    (I) We prove $\mathrm{ker}\,\overline{\partial}_0\subset \Omega_X^p(\log D)\otimes \mathcal{O}_X(F\otimes L)\otimes \mathscr{I}(h)$, i.e. for any $\sigma\in \mathscr{L}^{p,0}_{F\otimes L,h^F_Y\otimes h,\omega_Y}(W)$ such that $\overline{\partial}\sigma=0$, we get the $\sigma\in(\Omega_X^p(\log D)\otimes\mathcal{O}_X(F\otimes L)\otimes\mathscr{I}(h))(W)$.
    We can write 
    \begin{align*}
        \sigma(z)=\sum_{|I|=p,\lambda}\sigma_{I\lambda}(z)\zeta_I\otimes e_\lambda=\sum_{|I|=p,\lambda}\frac{\sigma_{I\lambda}(z)}{z_{I\cap\{1,\ldots,t\}}}dz_I\otimes e_\lambda
    \end{align*}
    where $\sigma_{I\lambda}(z)/z_{I\cap\{1,\ldots,t\}}$ is a measurable function. By the assumption $\overline{\partial}\sigma=0$ and Dolbeault-Grothendieck lemma (see [Dem-book,\,ChapterI]), the function $\sigma_{I\lambda}(z)/z_{I\cap\{1,\ldots,t\}}$ is holomorphic on $W^*_r$.
    Then the function $\sigma_{I\lambda}(z)=\sigma_{I\lambda}(z)/z_{I\cap\{1,\ldots,t\}}\cdot z_{I\cap\{1,\ldots,t\}}$ is also holomorphic on $W^*_r$.

    First, we show the case assuming condition $\nu(-\log h,x)<2$ for all points in $D$.
    If we denote $I\cap\{1,\ldots,t\}=\{i_{p1},\ldots,i_{pb}\}$, then we have that $\sigma\in \mathscr{L}^{p,0}_{F\otimes L,h^F_Y\otimes h,\omega_Y}(W_r) \iff$ 
    \begin{align*}
        +\infty>||\sigma||^2_{I^F_{\alpha,\tau}\otimes h,\omega_P}\Big|_{W^*_r}=&\int_{W^*_r}|\sigma|^2_{I^F_{\alpha,\tau}\otimes h,\omega_P}\omega_P^n\\
        =&\sum_{|I|=p,\lambda}\int_{W^*_r}|e_\lambda|^2_{I_F\otimes h}\Bigl(|\sigma_{I\lambda}(z)|^2\prod^b_{\nu=1}(\log|z_{i_{p\nu}}|^2)^2\prod^t_{j=1}|z_j|^{2\tau_j}(\log|z_j|^2)^{2\alpha}\Bigr)\omega_P^n\\
        \geq&\inf_Wh\sum_{|I|=p,\lambda}\int_{W^*_r}\Bigl(|\sigma_{I\lambda}(z)|^2\prod^b_{\nu=1}(\log|z_{i_{p\nu}}|^2)^2\prod^t_{j=1}|z_j|^{2\tau_j}(\log|z_j|^2)^{2\alpha}\Bigr)\omega_P^n\\
        \Longrightarrow \,~\mathrm{for ~any} \,\,I ~\mathrm{and}~\lambda, \,~&\int_{W^*_r}\Bigl(|\sigma_{I\lambda}(z)|^2\prod^b_{\nu=1}(\log|z_{i_{p\nu}}|^2)^2\prod^t_{j=1}|z_j|^{2\tau_j}(\log|z_j|^2)^{2\alpha}\Bigr)\omega_P^n<+\infty.
    \end{align*}
    Suppose that the Laurent series representation of $\sigma_{I\lambda}(z)$ on $W^*_r$ is given by 
    \begin{align*}
        \sigma_{I\lambda}(z)=\sum^\infty_{\beta=-\infty}\sigma_{I\lambda\beta}(z_{t+1},\ldots,z_n)z^{\beta_1}_1\cdots z^{\beta_t}_t, \quad \beta=(\beta_1,\ldots,\beta_t)
    \end{align*}
    where $\sigma_{I\lambda\beta}(z_{t+1},\ldots,z_n)$ is a holomorphic function on $\Delta^{n-t}_r$.
    Then by using polar coordinate, Fubini's theorem and Example \ref{Example of integral}, we see that if $\sigma$ is $L^2$-integrable on $W^*_r$ then $\beta_j>-\tau_j$ along $D_j$.
    In fact, 
    \begin{align*}
        +\infty&>\int_{W^*_r}\Bigl(|\sigma_{I\lambda}(z)|^2\prod^b_{\nu=1}(\log|z_{i_{p\nu}}|^2)^2\prod^t_{j=1}|z_j|^{2\tau_j}(\log|z_j|^2)^{2\alpha}\Bigr)\omega_P^n\\
        &=\int_{W^*_r}|\sigma_{I\lambda}(z)|^2\prod^t_{j=1}|z_j|^{2\tau_j}(\log|z_j|^2)^{2(\alpha+\sum^b_{\nu=1}\delta_{ji_{p\nu}})}\prod^t_{k=1}\Bigl(|z_k|^2(\log|z_k|^2)^2\Bigr)^{-1}dV_{\mathbb{C}^n}\\
        &=\int_{W^*_r}|\sigma_{I\lambda}(z)|^2\prod^t_{j=1}|z_j|^{2(\tau_j-1)}(\log|z_j|^2)^{2(\alpha+\sum^b_{\nu=1}\delta_{ji_{p\nu}}-1)}dV_{\mathbb{C}^n}\\
        &\geq\int_{W^*_r}|\sigma_{I\lambda\beta}(z)|^2\prod^t_{j=1}|z_j|^{2(\tau_j+\beta_j-1)}(\log|z_j|^2)^{2(\alpha+\sum^b_{\nu=1}\delta_{ji_{p\nu}}-1)}dV_{\mathbb{C}^n}\\
        &=c\int_{\Delta^t_r}\prod^t_{j=1}|z_j|^{2(\tau_j+\beta_j-1)}(\log|z_j|^2)^{2(\alpha+\sum^b_{\nu=1}\delta_{ji_{p\nu}}-1)}dV_{\mathbb{C}^t}\\
        &=c\prod^t_{j=1}2\pi\int^r_0r_j^{2(\tau_j+\beta_j-1)+1}(2\log r_j)^{2(\alpha+\sum^b_{\nu=1}\delta_{ji_{p\nu}}-1)}dr_j,
    \end{align*} 
    where $\omega^n_P=\Bigl(\prod^t_{k=1}|z_k|^2(\log|z_k|^2)^2\Bigr)^{-1}dV_{\mathbb{C}^n}$ and $c:=\int_{\Delta^{n-t}_r}|\sigma_{I\lambda\beta}|^2dV_{\mathbb{C}^{n-t}}>0$.

    Since $\tau_j\in (0,1]$, we obtain that $\beta_j\geq0$ and $\sigma_{I\lambda}(z)$ has removable singularity. 
    Hence $\sigma$ and $D\sigma$ have only logarithmic pole, and $\sigma$ is a section of $\Omega_X^p(\log D)\otimes \mathcal{O}_X(F\otimes L)=\Omega_X^p(\log D)\otimes \mathcal{O}_X(F\otimes L)\otimes\mathscr{I}(h)$ on $W_r$, 
    where $\mathcal{O}_X=\mathscr{I}(h)$ near $x_0$ by the assumption $\nu(-\log h,x_0)<2$ and Skoda's result.

    Second, we show the case without the Lelong number condition of $-\log h$, i.e. case $\tau_j=1$.
    Here, we already know that for any $\sigma\in \mathscr{L}^{p,0}_{F\otimes L,h^F_Y\otimes h,\omega_Y}(W)$ such that $\overline{\partial}\sigma=0$, we get $\sigma\in(\Omega_X^p(\log D)\otimes\mathcal{O}_X(F\otimes L))(W)$ by the above.
    It is sufficient to show that $\sigma\in(\Omega_X^p(\log D)\otimes\mathcal{O}_X(F\otimes L)\otimes\mathscr{I}(h))(W)$, i.e. $\sigma_{I\lambda}\in\mathscr{I}(h)(W)$ for any $I$ and $\lambda$.

    From the above inequality, if $r\in(0,1/2)$ then we have that 
    \begin{align*}
        +\infty>||\sigma||^2_{I^F_\alpha\otimes h,\omega_P}\Big|_{W^*_r} 
        &=\sum_{|I|=p,\lambda}\int_{W^*_r}\Bigl(|\sigma_{I\lambda}(z)|_h^2\prod^b_{\nu=1}(\log|z_{i_{p\nu}}|^2)^2\prod^t_{j=1}|z_j|^{2}(\log|z_j|^2)^{2\alpha}\Bigr)\omega_P^n\\
        &=\sum_{|I|=p,\lambda}\int_{W_r}|\sigma_{I\lambda}(z)|_h^2\prod^t_{j=1}(\log|z_j|^2)^{2(\alpha+\sum^b_{\nu=1}\delta_{ji_{p\nu}}-1)}dV_{\mathbb{C}^n}\\
        &\geq\sum_{|I|=p,\lambda}\int_{W_r}|\sigma_{I\lambda}(z)|_h^2dV_{\mathbb{C}^n},
    \end{align*}
    where $I^F_\alpha:=I^F_{\alpha,1}=I_F\prod^t_{j=1}|z_j|^{2}(\log|z_j|^2)^{2\alpha}$ on $W^*_r$ by $\tau_j=1$ and $\inf_{W_r}(\log|z_j|^2)^{2\alpha}>1$. 

    (II) We prove $\Omega_X^p(\log D)\otimes \mathcal{O}_X(F\otimes L)\otimes \mathscr{I}(h)\subset \mathrm{ker}\,\overline{\partial}_0$, i.e. for any $\sigma\in(\Omega_X^p(\log D)\otimes\mathcal{O}_X(F\otimes L)\otimes\mathscr{I}(h))(W)$, we get $\sigma\in \mathscr{L}^{p,0}_{F\otimes L,h^F_Y\otimes h,\omega_Y}(W)$.
    Here, $\sigma(z)=\sum_{|I|=p,\lambda}\sigma_{I\lambda}(z)\zeta_I\otimes e_\lambda=\sum_{|I|=p,\lambda}\frac{\sigma_{I\lambda}(z)}{z_{I\cap\{1,\ldots,t\}}}dz_I\otimes e_\lambda$, any $\sigma_{I\lambda}$ is holomorphic on $W$ and $\int_W|\sigma_{I\lambda}|^2_hdV_{\mathbb{C}^n}<+\infty$ by the assumption.
    Then the $L^2$-norm of $\sigma$ is the following:
    \begin{align*}
        ||\sigma||^2_{I^F_{\alpha,\tau}\otimes h,\omega_P}\Big|_{W^*_r}
        =\sum_{|I|=p,\lambda}\int_{W^*_r}|\sigma_{I\lambda}(z)|^2_h\prod^t_{j=1}|z_j|^{2(\tau_j-1)}(\log|z_j|^2)^{2(\alpha+\sum^b_{\nu=1}\delta_{ji_{p\nu}}-1)}dV_{\mathbb{C}^n}.
    \end{align*}
    
    We show that for any $I$ and $\lambda$, 
    \begin{align*}
        J_{I,\lambda}:=\int_{W^*_r}|\sigma_{I\lambda}(z)|^2_h\prod^t_{j=1}|z_j|^{2(\tau_j-1)}(\log|z_j|^2)^{2(\alpha+\sum^b_{\nu=1}\delta_{ji_{p\nu}}-1)}dV_{\mathbb{C}^n}<+\infty.
    \end{align*}
    Let $f_p:=|\sigma_{I\lambda}|^{2/p}h=|\sigma_{I\lambda}|^{2/p}e^{-\varphi}$ and $g=|\sigma_{I\lambda}|^{2(1-1/p)}\prod^t_{j=1}|z_j|^{2(\tau_j-1)}(-\log|z_j|^2)^{2(\alpha+\sum^b_{\nu=1}\delta_{ji_{p\nu}}-1)}$ on $W^*_r$ for any $p>1$. 
    By the H\"older's inequality, we get 
    \begin{align*}
        J_{I,\lambda}=\int_{W^*_r}|f_pg|dV_{\mathbb{C}^n}\leq\Bigl(\int_{W^*_r}|f_p|^pdV_{\mathbb{C}^n}\Bigr)^{1/p}\cdot\Bigl(\int_{W^*_r}|g|^qdV_{\mathbb{C}^n}\Bigr)^{1/q}
    \end{align*}
    Here $1=1/p+1/q$. 
    From $\sigma_{I\lambda}\in\mathscr{I}(\varphi)(W)$ and the strong openness property on line bundles (see \cite{GZ15}, \cite{JM12}), for some $r'\in(0,r)$ there exists $p>1$ such that $\int_{W_{r'}}|f_p|^pdV_{\mathbb{C}^n}=\int_{W_{r'}}|\sigma_{I\lambda}|^2e^{-p\varphi}dV_{\mathbb{C}^n}<+\infty$. And we get
    \begin{align*}
        \int_{W^*_r}|g|^qdV_{\mathbb{C}^n}
        &=\int_{W^*_r}|\sigma_{I\lambda}|^2\prod^t_{j=1}|z_j|^{2q(\tau_j-1)}(-\log|z_j|^2)^{2q(\alpha+\sum^b_{\nu=1}\delta_{ji_{p\nu}}-1)}dV_{\mathbb{C}^n}\\
        &\leq C(\pi r^2)^{n-t+1}\prod^t_{j=1}\int_{|z_j|<r}|z_j|^{2q(\tau_j-1)}(-\log|z_j|^2)^{2q(\alpha+\sum^b_{\nu=1}\delta_{ji_{p\nu}}-1)}dz_j\wedge d\overline{z}_j\\
        &=C(\pi r^2)^{n-t+1}\prod^t_{j=1}2\pi\int^r_0r_j^{2q(\tau_j-1)+1}(-\log r_j^2)^{2q(\alpha+\sum^b_{\nu=1}\delta_{ji_{p\nu}}-1)}dr_j,
    \end{align*}
    where $C:=\sup_{W_r}|\sigma_{I\lambda}|^2<+\infty$. Hence, by Example \ref{Example of integral} if $\tau_j>1-1/q=1/p$, then $\int_{W^*_r}|g|^qdV_{\mathbb{C}^n}<+\infty$ and $\sigma$ is $L^2$-integrable on $W^*_{r'}$.
    This concludes the case without the Lelong number condition of $-\log h$, i.e. the case of $\tau_j=1$.

    Finally, we show the case assuming condition $\nu(-\log h,x)<2\delta$ for all points in $D$.
    Then $h^{1/\delta}=e^{-\varphi/\delta}$ is locally integrable around $x_0\in D$ by Skoda's result.
    Let $g=|\sigma_{I\lambda}|^2\prod^t_j|z_j|^{2(\tau_j-1)}(-\log|z_j|^2)^{2(\alpha+\sum^b_{\nu}\delta_{ji_{p\nu}}-1)}$ on $W^*_r$.
    By H\"older's inequality, we get 
    \begin{align*}
        J_{I,\lambda}=\int_{W^*_r}|hg|dV_{\mathbb{C}^n}\leq\Bigl(\int_{W^*_r} h^{1/\delta}dV_{\mathbb{C}^n}\Bigr)^{\delta}\cdot\Bigl(\int_{W^*_r}|g|^qdV_{\mathbb{C}^n}\Bigr)^{1/q},
    \end{align*}
    where $q=1/(1-\delta)$. Similarly to the above,
    \begin{align*}
        \int_{W^*_r}|g|^qdV_{\mathbb{C}^n}\leq C(\pi r^2)^{n-t+1}\prod^t_{j=1}2\pi\int^r_0r_j^{2q(\tau_j-1)+1}(-\log r_j^2)^{2q(\alpha+\sum^b_{\nu=1}\delta_{ji_{p\nu}}-1)}dr_j,
    \end{align*}
    where $C:=\sup_{W_r}|\sigma_{I\lambda}|^{2q}<+\infty$.
    Hence, by Example \ref{Example of integral} if $\tau_j>1-1/q=\delta$, then $\int_{W^*_r}|g|^qdV_{\mathbb{C}^n}<+\infty$ and $\sigma$ is $L^2$-integrable on $W^*_{r}$, i.e. $\sigma\in \mathscr{L}^{p,0}_{F\otimes L,h^F_Y\otimes h,\omega_Y}(W)$.
\end{proof}

\vspace{2mm}

$\mathit{Proof ~of ~Theorem ~\ref{L2-type Log Dolbeault isomorphism with sHm line}}$.
    From Proposition \ref{L2-type Log Dol isom at q=0 with sHm line}, it is sufficient to show the exactness of $(\ast)$ on $D$ at $q\geq1$. 
    For any fixed $r\in(0,1)$, we define a new smooth Hermitian metric $h^F_{\alpha,\tau}$ on $F$ over $W^*_r$ as
    \begin{align*}
        h^F_{\alpha,\tau}=h_Fe^{-3\alpha|z|^2}\prod^t_{j=1}|z_j|^{2\tau_j}(\log|z_j|^2)^{2\alpha},
    \end{align*}
    where $|z|^2=\sum^n_{j=1}|z_j|^2$. Then we have that $h^F_Y\thicksim h^F_{\alpha,\tau}$ on $W^*_r$.

\begin{lemma}
    The Chern curvature of $h^F_{\alpha,\tau}$ satisfies
    \begin{align*}
        i\Theta_{F,h^F_{\alpha,\tau}}\geq_{Nak}2\alpha\omega_P\otimes\mathrm{id}_F
    \end{align*} 
    on $W^*_r$ for some large $\alpha>0$.
\end{lemma}

\begin{proof}
    It is easy to show that for any $1\leq j\leq t$,
    \begin{align*}
        \idd\log|z_j|^2=0 \quad \mathrm{and}\quad -\idd\log(\log|z_j|^2)^2=\frac{2idz_j\wedge d\overline{z}_j}{|z_j|^2(\log|z_j|^2)^2}
    \end{align*}
    on $W^*_r$. The curvature of $(F,h^F_{\alpha,\tau})$ is given by 
    \begin{align*}
        i\Theta_{F,h^F_{\alpha,\tau}}&=\Bigl(-i\sum^t_{j=1}\tau_j\partial\overline{\partial}\log|z_j|^2-i\alpha\sum^t_{j=1}\partial\overline{\partial}\log(\log|z_j|^2)^2+3i\alpha\sum^n_{j=1}\partial\overline{\partial}|z_j|^2\Bigr)\otimes\mathrm{id}_F+i\Theta_{F,h_F}\\
        &=\Bigl(2i\alpha\sum^t_{j=1}\frac{dz_j\wedge d\overline{z}_j}{|z_j|^2(\log|z_j|^2)^2}+3i\alpha\sum^n_{j=1}\partial\overline{\partial}|z_j|^2\Bigr)\otimes\mathrm{id}_F+i\Theta_{F,h_F}\\
        &\geq_{Nak} 2i\alpha\Bigl(\sum^t_{j=1}\frac{dz_j\wedge d\overline{z}_j}{|z_j|^2(\log|z_j|^2)^2}+\sum^n_{j=t+1}\partial\overline{\partial}|z_j|^2\Bigr)\otimes\mathrm{id}_F=2\alpha\omega_P\otimes\mathrm{id}_F,
    \end{align*}
    if we choose $\alpha$ large enough so that $\Bigl(i\alpha\sum^n_{j=1}\partial\overline{\partial}|z_j|^2\Bigr)\otimes\mathrm{id}_F+i\Theta_{F,h_F}\geq_{Nak}0$ on $W^*_r$.
\end{proof}

\begin{lemma}\label{Lemma positivity of h_V}$(\mathrm{cf.\,[HLWY16,\,Lemma\,3.3]})$
    On the chart $W^*_r$, the holomorphic vector bundle $V:=\Omega^p_Y\otimes K^{-1}_Y\otimes F|_Y$ with the induced metric $h_V$ by $\omega_P$ and $h^F_{\alpha,\tau}$ is Nakano positive where $\alpha$ is large enough.

    Moreover, for any $u\in L^2_{n,q}(W^*_r,V,h_V,\omega_P)\cong L^2_{p,q}(W^*_r,F|_Y,h^F_{\alpha,\tau},\omega_P)$ we have 
    \begin{align*}
        \langle[i\Theta_{V,h_V},\Lambda_{\omega_P}]u,u\rangle_{h_V,\omega_P}\geq C_q|u|^2_{h_V,\omega_P},\,\,\, i.e.\,\,\,A^{n,q}_{V,h_V,\omega_P}\geq C_q,
    \end{align*}
    where $C_q$ is a positive constant independent of $q$. Therefore $h_V$ is Nakano positive.
\end{lemma}

\begin{lemma}
    The sequence of $(\ast)$ on $D$ is exact at $q\geq1$. This is, on a small local chart $W^*_r=(\Delta^*_r)^t\times(\Delta_r)^{n-t}$, for any $\overline{\partial}$-closed $f\in L^2_{p,q}(W^*_r,F\otimes L,h^F_{\alpha,\tau}\otimes h,\omega_P)$,
    there exists $u\in L^2_{p,q-1}(W^*_r,F\otimes L,h^F_{\alpha,\tau}\otimes h,\omega_P)$ such that $\overline{\partial}u=f$.
\end{lemma}

\begin{proof}
    We can take $r$ so that $F,L$ and $T_Y$ are trivial over $W^*_r$. 
    Then we get 
    \begin{align*}
        L^2_{p,q}(W^*_r,F\otimes L,h^F_{\alpha,\tau}\otimes h,\omega_P)\cong L^2_{n,q}(W^*_r,V\otimes L,h_V\otimes h,\omega_P). 
    \end{align*}

    By Steinness of $W^*_r$, the set $W^*_r$ has a complete \kah metric.
    Since Nakano positivity of $h_V$ and [Wat22b,\,Theorem\,4.3], there exists $u\in L^2_{n,q-1}(W^*_r,V\otimes L,h_V\otimes h,\omega_P)\cong L^2_{p,q-1}(W^*_r,F\otimes L,h^F_{\alpha,\tau}\otimes h,\omega_P)$ such that $\overline{\partial}u=f$ and
    \begin{align*}
        \int_{W^*_r}|u|^2_{h_V\otimes h,\omega_P}dV_{\omega_P}\leq\int_{W^*_r}\langle B_{h_V,\omega_P}^{-1}f,f\rangle_{h_V\otimes h,\omega_P}dV_{\omega_P}\leq\frac{1}{C_q}\int_{W^*_r}|f|^2_{h_V\otimes h}dV_{\omega_P}<+\infty,
    \end{align*}
    where $B_{h_V,\omega_P}:=[i\Theta_{V,h_V}\otimes\mathrm{id}_L,\Lambda_{\omega_P}]\geq C_q$ by Lemma \ref{Lemma positivity of h_V}.
    Hence, the sequence of $(\ast)$ on $D$ is exact at $q\geq1$.
\end{proof}


Given the $L^2$-type Log Dolbeault sheaf resolution $(\ast)$, the following $L^2$-type Log Dolbeault isomorphisms
\begin{align*}
    H^q(X,\Omega_X^p(\log D)\otimes F\otimes L \otimes \mathscr{I}(h))\cong H^q(\Gamma(X,\mathscr{L}^{p,\ast}_{F\otimes L,h^F_Y\otimes h,\omega_P})). 
\end{align*}
are clear. The proof of Theorem \ref{L2-type Log Dolbeault isomorphism with sHm line} is completed. \qed

\subsection{An $L^2$-type Dolbeault isomorphism for logarithmic sheaves twisted by $L^2$-subsheaves} 

We prove the following $L^2$-type Dolbeault sheaf resolution using the same approach as in the previous subsection.

\begin{theorem}\label{L2-type Log Dolbeault isomorphism with sHm vector}
    Let $(X,\omega)$ be a compact \kah manifold, $D=\sum^s_{j=1}D_j$ be a simple normal crossing divisor in $X$ and $\omega_P$ be a smooth \kah metric on $Y:=X\setminus D$ which is of Poincar\'e type along $D$. 
    Let $(F,h_F)$ be a holomorphic vector bundle over $X$, $E$ be a holomorphic vector bundle over $X$ equipped with a singular Hermitian metric $h$ and $\sigma_j$ be the defining section of $D_j$. Fix smooth Hermitian metrics $||\bullet||_{D_j}$ on $\mathcal{O}_X(D_j)$.
    We assume that $h$ is Griffiths semi-positive.

    Then for any enough large integer $\alpha>0$ there exists a smooth Hermitian metric 
    \begin{align*}
        h^F_Y=h_F\prod^s_{j=1}||\sigma_j||^2_{D_j}(\log ||\sigma_j||^2_{D_j})^{2\alpha}
    \end{align*}
    on $F|_Y$ such that the sheaf $\Omega_X^p(\log D) \otimes\mathcal{O}_X(F)\otimes \mathscr{E}(h\otimes\mathrm{det}\,h)$ over $X$ enjoys a fine resolution 
    given by the $L^2$-Dolbeault complex $(\mathscr{L}^{p,\ast}_{F\otimes E\otimes\mathrm{det}\,E,h^F_Y\otimes h\otimes\mathrm{det}\,h,\omega_P},\overline{\partial})$.

    Moreover, if there exists $0<\delta\leq1$ such that $\nu(-\log\mathrm{det}\,h,x)<\delta$ for all points in $D$ then for any enough large integer $\alpha>0$ and any $\tau_j\in(\delta,1]$ there exists a smooth Hermitian metric 
    \begin{align*}
        h^F_Y=h_F\prod^s_{j=1}||\sigma_j||^{2\tau_j}_{D_j}(\log ||\sigma_j||^2_{D_j})^{2\alpha}
    \end{align*}
    on $F|_Y$ such that the sheaf $\Omega_X^p(\log D) \otimes\mathcal{O}_X(F)\otimes \mathscr{E}(h\otimes\mathrm{det}\,h)$ over $X$ enjoys a fine resolution 
    given by the $L^2$-Dolbeault complex $(\mathscr{L}^{p,\ast}_{F\otimes E\otimes\mathrm{det}\,E,h^F_Y\otimes h\otimes\mathrm{det}\,h,\omega_P},\overline{\partial})$.

    In other words, we have the $L^2$-type Dolbeault sheaf resolution:
    \begin{align*}
        0\longrightarrow \Omega_X^p(\log D)\otimes \mathcal{O}_X(F)\otimes \mathscr{E}(h\otimes\mathrm{det}\,h) \longrightarrow \mathscr{L}^{p,\ast}_{F\otimes E\otimes\mathrm{det}\,E,h^F_Y\otimes h\otimes\mathrm{det}\,h,\omega_P} 
    \end{align*}
    where $\mathscr{L}^{p,q}_{F\otimes E\otimes\mathrm{det}\,E,h^F_Y\otimes h\otimes\mathrm{det}\,h,\omega_P}$ is a fine sheaf for any $0\leq p,q\leq n$. In particular,
    \begin{align*}
        H^q(X,\Omega_X^p(\log D)\otimes F\otimes \mathscr{E}(h\otimes\mathrm{det}\,h))\cong H^q(\Gamma(X,\mathscr{L}^{p,\ast}_{F\otimes E\otimes\mathrm{det}\,E,h^F_Y\otimes h\otimes\mathrm{det}\,h,\omega_P})). 
    \end{align*}
\end{theorem}

\begin{proposition}\label{L2-type Log Dol isom at q=0 with sHm vector}
    Let $X,Y,D,\omega,\omega_P,F,h_F,E,h,\sigma_j$ and $||\bullet||_{D_j}$ be as in above.
    We assume that $h$ is Griffiths semi-positive. 

    Then for any integer $\alpha>0$ there exists a smooth Hermitian metric 
    \begin{align*}
        h^F_Y=h_F\prod^s_{j=1}||\sigma_j||^2_{D_j}(\log ||\sigma_j||^2_{D_j})^{2\alpha}
    \end{align*}
    on $F|_Y$ such that the following $L^2$-Dolbeault complex $(\mathscr{L}^{p,\ast}_{F\otimes E\otimes\mathrm{det}\,E,h^F_Y\otimes h\otimes\mathrm{det}\,h,\omega_P},\overline{\partial})$ over $X$ is exact at $q=0$
    \begin{align*}
        0\longrightarrow \Omega_X^p(\log D)\otimes \mathcal{O}_X(F)\otimes \mathscr{E}(h\otimes\mathrm{det}\,h) \longrightarrow \mathscr{L}^{p,\ast}_{F\otimes E\otimes\mathrm{det}\,E,h^F_Y\otimes h\otimes\mathrm{det}\,h,\omega_P}, 
    \end{align*}
    i.e. $\mathrm{ker}\,\overline{\partial}_0=\Omega_X^p(\log D)\otimes \mathcal{O}_X(F)\otimes \mathscr{E}(h\otimes\mathrm{det}\,h)$.

    Moreover, if there exists $0<\delta\leq1$ such that $\nu(-\log\mathrm{det}\,h,x)<\delta$ for all points in $D$ then for any integer $\alpha>0$ and any $\tau_j\in(\delta,1]$ there exists a smooth Hermitian metric 
    \begin{align*}
        h^F_Y=h_F\prod^s_{j=1}||\sigma_j||^{2\tau_j}_{D_j}(\log ||\sigma_j||^2_{D_j})^{2\alpha}
    \end{align*}
    on $F|_Y$ such that $L^2$-Dolbeault complex $(\mathscr{L}^{p,\ast}_{F\otimes E\otimes\mathrm{det}\,E,h^F_Y\otimes h\otimes\mathrm{det}\,h,\omega_P},\overline{\partial})$ is exact at $q=0$.
\end{proposition}

Similarly to the line bundle case in $\S 3.1$, from Theorem \ref{L2-type Dolbeault isomorphism with sHm}, 
it is sufficient that Theorem \ref{L2-type Log Dolbeault isomorphism with sHm vector} and Proposition \ref{L2-type Log Dol isom at q=0 with sHm vector} show the exactness of the $L^2$-Dolbeault complex $(\mathscr{L}^{p,\ast}_{F\otimes E\otimes\mathrm{det}\,E,h^F_Y\otimes h\otimes\mathrm{det}\,h,\omega_P},\overline{\partial})$ for all points on $D$.

\begin{proof}
    The approach to the proof is the same as in Proposition \ref{L2-type Log Dol isom at q=0 with sHm line}.
    For any fixed point $x_0\in D$, let $(W;z_1,\ldots,z_n)$ be a local coordinate chart centered at $x_0$ along $D$ such that the locus of $D$ is given by $z_1\cdots z_t=0$ and that $F$ is trivial, i.e. $F|_W=W\times\mathbb{C}^r:=\underline{\mathbb{C}^r}$.
    Here, $W=W_r:=\Delta^n_r$, $Y\cap W=W\setminus D:=W^*_r=(\Delta^*_r)^t\times(\Delta_r)^{n-t}$ and $r\in(0,1)$.
    By Griffiths semi-positivity of $h$, the function $-\log\mathrm{det}\,h$ is plurisubharmonic on $W$ and $\sup_W-\log\mathrm{det}\,h<+\infty$, i.e. $\inf_W\mathrm{det}\,h>0$. 
    Let $I_F$ be a trivial Hermitian metric of $F=\underline{\mathbb{C}^r}$ on $W$ and $b=(b_1,\ldots,b_r), f=(e_1,\ldots,e_l)$ be holomorphic frames of $F, E$ on $W$ respectively where $b_j=(0,\cdots,0,1,0,\cdots,0)$ and $b$ is orthonormal with respect to $I_F$. 

    Note that the smooth Hermitian metric $h^F_{\alpha,\tau}|_{W^*_r}$ is equivalent to the following smooth Hermitian metric
    $I^F_{\alpha,\tau}:=I_F\prod^t_{j=1}|z_j|^{2\tau_j}(\log|z_j|^2)^{2\alpha}$ on $W^*_r$,
    where $h_F\sim I_F$ and that $\mathscr{L}^{p,q}_{F\otimes E\otimes\mathrm{det}\,E,h^F_Y\otimes h\otimes\mathrm{det}\,h,\omega_Y}(W)= \mathscr{L}^{p,q}_{\underline{\mathbb{C}^r}\otimes E\otimes\mathrm{det}\,E,h^F_\alpha\otimes h\otimes\mathrm{det}\,h,\omega_Y}(W)$.
    Denote $e_{det}:=\wedge^l_{j=1}e_j$, $\zeta_j=\frac{1}{z_j}dz_j$ for $1\leq j\leq t$ and $\zeta_j=dz_j$ for $t+1\leq j\leq n$.

    (I) We prove $\mathrm{ker}\,\overline{\partial}_0\subset \Omega_X^p(\log D)\otimes \mathcal{O}_X(F)\otimes \mathscr{E}(h\otimes\mathrm{det}\,h)$, i.e. for any $\sigma\in \mathscr{L}^{p,0}_{F\otimes E\otimes\mathrm{det}\,E,}$ $_{h^F_Y\otimes h\otimes\mathrm{det}\,h,\omega_Y}(W)$ such that $\overline{\partial}\sigma=0$, we get $\sigma\in(\Omega_X^p(\log D)\otimes\mathcal{O}_X(F)\otimes\mathscr{E}(h\otimes\mathrm{det}\,h))(W)$.
    We can write 
    \begin{align*}
        \sigma(z)&=\sum_{|I|=p,\lambda,\mu}\sigma_{I\lambda\mu}(z)\zeta_I\otimes b_\lambda\otimes e_\mu\otimes e_{det}=\sum_{|I|=p,\lambda}\sigma_{I\lambda}(z)\zeta_I\otimes b_\lambda\\
        &=\sum_{|I|=p,\lambda}\frac{\sigma_{I\lambda\mu}(z)}{z_{I\cap\{1,\ldots,t\}}}dz_I\otimes b_\lambda\otimes e_\mu\otimes e_{det},
    \end{align*}
    where $\sigma_{I\lambda}(z)=\sum_\mu \sigma_{I\lambda\mu}(z)e_\mu\otimes e_{det}\in H^0(W,E\otimes\mathrm{det}\,E)$. 
    By Dolbeault-Grothendieck lemma (see [Dem-book,\,ChapterI]), $\sigma_{I\lambda\mu}(z)/z_{I\cap\{1,\ldots,t\}}$ is a holomorphic function on $W^*_r$.
    Then $\sigma_{I\lambda\mu}(z)$ is a holomorphic function on $W^*_r$.

    First, we show the case assuming $\nu(-\log\mathrm{det}\,h,x)<2$ for all points in $D$.
    If we denote $I\cap\{1,\ldots,t\}=\{i_{p1},\ldots,i_{pb}\}$, then we have that $\sigma\in \mathscr{L}^{p,0}_{F\otimes E\otimes\mathrm{det}\,E,h^F_Y\otimes h\otimes\mathrm{det}\,h,\omega_Y}(W_r) \iff $ 
    \begin{align*}
        +\infty>&||\sigma||^2_{I^F_{\alpha,\tau}\otimes h\otimes\mathrm{det}\,h,\omega_P}\Big|_{W^*_r}=\int_{W^*_r}|\sigma|^2_{I^F_{\alpha,\tau}\otimes h\otimes\mathrm{det}\,h,\omega_P}\omega_P^n\\
        =&\sum_{I,\lambda}\int_{W^*_r}|\sigma_{I\lambda}(z)|^2_{h\otimes\mathrm{det}\,h}\prod^t_{j=1}|z_j|^{2(\tau_j-1)}(\log|z_j|^2)^{2(\alpha+\sum^b_{\nu=1}\delta_{ji_{p\nu}}-1)}dV_{\mathbb{C}^n},
    \end{align*}
    where $\omega^n_P=\Bigl(\prod^t_{k=1}|z_k|^2(\log|z_k|^2)^2\Bigr)^{-1}dV_{\mathbb{C}^n}$. Then for any $I$ and $\lambda$, we get 
    \begin{align*}
        \int_{W^*_r}|\sigma_{I\lambda}(z)|^2_{h\otimes\mathrm{det}\,h}\prod^t_{j=1}|z_j|^{2(\tau_j-1)}(\log|z_j|^2)^{2(\alpha+\sum^b_{\nu=1}\delta_{ji_{p\nu}}-1)}dV_{\mathbb{C}^n}<+\infty.
    \end{align*}

    Since Griffiths semi-positivity of $h$, there exists a sequence of smooth Griffiths semi-positive Hermitian metrics $(h_\nu)_{\nu\in\mathbb{N}}$ increasing to $h$ a.e. pointwise on $W$ (see [Wat22b, \,Proposition\,3.15]).
    Therefore, for any $I,\lambda$ and $\mu$, we have that 
    \begin{align*}
        +\infty&>\int_{W^*_r}|\sigma_{I\lambda}(z)|^2_{h\otimes\mathrm{det}\,h}\prod^t_{j=1}|z_j|^{2(\tau_j-1)}(\log|z_j|^2)^{2(\alpha+\sum^b_{\nu=1}\delta_{ji_{p\nu}}-1)}dV_{\mathbb{C}^n}\\
        &\geq\int_{W^*_r}|\sigma_{I\lambda}(z)|^2_{h_\nu\otimes\mathrm{det}\,h_\nu}\prod^t_{j=1}|z_j|^{2(\tau_j-1)}(\log|z_j|^2)^{2(\alpha+\sum^b_{\nu=1}\delta_{ji_{p\nu}}-1)}dV_{\mathbb{C}^n}\\
        &=\int_{W^*_r}\Bigl(\sum_{k,l}\overline{\sigma}_{I\lambda k}(z)h_\nu^{kl}\sigma_{I\lambda l}(z)\Bigr)\mathrm{det}\,h_\nu\prod^t_{j=1}|z_j|^{2(\tau_j-1)}(\log|z_j|^2)^{2(\alpha+\sum^b_{\nu=1}\delta_{ji_{p\nu}}-1)}dV_{\mathbb{C}^n}\\
        &\geq\int_{W^*_r}h^{\mu\mu}_\nu\mathrm{det}\,h_\nu|\sigma_{I\lambda\mu}(z)|^2\prod^t_{j=1}|z_j|^{2(\tau_j-1)}(\log|z_j|^2)^{2(\alpha+\sum^b_{\nu=1}\delta_{ji_{p\nu}}-1)}dV_{\mathbb{C}^n}\\
        &\geq C\int_{W^*_r}|\sigma_{I\lambda\mu}(z)|^2\prod^t_{j=1}|z_j|^{2(\tau_j-1)}(\log|z_j|^2)^{2(\alpha+\sum^b_{\nu=1}\delta_{ji_{p\nu}}-1)}dV_{\mathbb{C}^n},
    \end{align*}
    where $h_\nu=(h^{kl}_\nu)_{1\leq k,l\leq\mathrm{rank}\,E}$ and $C:=(\inf_{W_r}h^{\mu\mu}_\nu)(\inf_{W_r}\mathrm{det}\,h_\nu)>0$. 

    Suppose that the Laurent series representation of $\sigma_{I\lambda\mu}(z)$ on $W^*_r$ is given by 
    \begin{align*}
        \sigma_{I\lambda\mu}(z)=\sum^\infty_{\beta=-\infty}\sigma_{I\lambda\mu\beta}(z_{t+1},\ldots,z_n)z^{\beta_1}_1\cdots z^{\beta_t}_t, \quad \beta=(\beta_1,\ldots,\beta_t)
    \end{align*}
    where $\sigma_{I\lambda\mu\beta}(z_{t+1},\ldots,z_n)$ is a holomorphic function on $\Delta^{n-t}_r$.
    Then by using polar coordinate, Fubini's theorem and Example \ref{Example of integral}, we see that if $\sigma$ is $L^2$-integrable on $W^*_r$ then $\beta_j>-\tau_j$ along $D_j$.
    In fact, 
    \begin{align*}
        +\infty&>\int_{W^*_r}|\sigma_{I\lambda\mu}(z)|^2\prod^t_{j=1}|z_j|^{2(\tau_j-1)}(\log|z_j|^2)^{\alpha+2(\sum^b_{\nu=1}\delta_{ji_{p\nu}}-1)}dV_{\mathbb{C}^n}\\
        &\geq\int_{W^*_r}|\sigma_{I\lambda\mu\beta}(z)|^2\prod^t_{j=1}|z_j|^{2(\tau_j+\beta_j-1)}(\log|z_j|^2)^{\alpha+2(\sum^b_{\nu=1}\delta_{ji_{p\nu}}-1)}dV_{\mathbb{C}^n}\\
        &= c\int_{\Delta^t_r}\prod^t_{j=1}|z_j|^{2(\tau_j+\beta_j-1)}(\log|z_j|^2)^{\alpha+2(\sum^b_{\nu=1}\delta_{ji_{p\nu}}-1)}dV_{\mathbb{C}^t}\\
        &=c\prod^t_{j=1}2\pi\int^r_0r_j^{2(\tau_j+\beta_j-1)+1}\prod^b_{\nu=1}(2\log r_j)^{\alpha+2(\sum^b_{\nu=1}\delta_{ji_{p\nu}}-1)}dr_j,
    \end{align*} 
    where $c=\int_{\Delta^{n-t}_r}|\sigma_{I\lambda\mu\beta}|^2dV_{\mathbb{C}^{n-t}}>0$.

    Since $\tau_j\in (0,1]$, we obtain that $\beta_j\geq0$ and $\sigma_{I\lambda\mu}(z)$ has removable singularity. 
    Hence $\sigma$ and $D\sigma$ have only logarithmic pole, and $\sigma$ is a section of $\Omega_X^p(\log D)\otimes \mathcal{O}_X(F\otimes E\otimes\mathrm{det}\,E)=\Omega_X^p(\log D)\otimes \mathcal{O}_X(F)\otimes\mathscr{E}(h\otimes\mathrm{det}\,h)$ on $W_r$, 
    where $\mathcal{O}_X(E\otimes\mathrm{det}\,E)=\mathscr{E}(h\otimes\mathrm{det}\,h)$ near $x_0$ by the assumption $\nu(-\log\mathrm{det}\,h,x_0)<2$ and Skoda's result.

    Second, we show the case without the Lelong number condition, i.e. case $\tau_j=1$.
    Here, we already know that for any $\sigma\in \mathscr{L}^{p,0}_{F\otimes E\otimes\mathrm{det}\,E,h^F_Y\otimes h\otimes\mathrm{det}\,h,\omega_Y}(W)$ such that $\overline{\partial}\sigma=0$, we get $\sigma\in(\Omega_X^p(\log D)\otimes\mathcal{O}_X(F\otimes L))(W)$ by the above.
    It is sufficient that $\sigma\in(\Omega_X^p(\log D)\otimes\mathcal{O}_X(F)\otimes\mathscr{E}(h\otimes\mathrm{det}\,h))(W)$, i.e. $\sigma_{I\lambda}\in\mathscr{E}(h\otimes\mathrm{det}\,h)(W)$ for any $I$ and $\lambda$.

    From the above inequality, if $r\in(0,1/2)$ then we have that 
    \begin{align*}
        +\infty>||\sigma||^2_{I^F_\alpha\otimes h\otimes\mathrm{det}\,h,\omega_P}\Big|_{W^*_r}
        &=\sum_{I,\lambda}\int_{W^*_r}|\sigma_{I\lambda}(z)|^2_{h\otimes\mathrm{det}\,h}\prod^t_{j=1}(\log|z_j|^2)^{2(\alpha+\sum^b_{\nu=1}\delta_{ji_{p\nu}}-1)}dV_{\mathbb{C}^n}\\
        &\geq\sum_{I,\lambda}\int_{W_r}|\sigma_{I\lambda}(z)|^2_{h\otimes\mathrm{det}\,h}dV_{\mathbb{C}^n},
    \end{align*}
    where $I^F_\alpha:=I^F_{\alpha,1}=I_F\prod^t_{j=1}|z_j|^{2}(\log|z_j|^2)^{2\alpha}$ on $W^*_r$ by $\tau_j=1$ and $\inf_{W_r}(\log|z_j|^2)^{2\alpha}>1$. 

    (II) We prove $\Omega_X^p(\log D)\otimes \mathcal{O}_X(F)\otimes \mathscr{E}(h\otimes\mathrm{det}\,h)\subset \mathrm{ker}\,\overline{\partial}_0$, i.e. for any $\sigma\in(\Omega_X^p(\log D)\otimes\mathcal{O}_X(F)\otimes\mathscr{E}(h\otimes\mathrm{det}\,h))(W)$, we get $\sigma\in \mathscr{L}^{p,0}_{F\otimes E\otimes\mathrm{det}\,E,h^F_Y\otimes h\otimes\mathrm{det}\,h,\omega_Y}(W)$.
    Here, $\sigma(z)=\sum_{|I|=p,\lambda,\mu}\sigma_{I\lambda\mu}(z)\zeta_I\otimes b_\lambda\otimes e_\mu\otimes e_{det}=\sum_{|I|=p,\lambda,\mu}\frac{\sigma_{I\lambda\mu}(z)}{z_{I\cap\{1,\ldots,t\}}}dz_I\otimes b_\lambda\otimes e_\mu\otimes e_{det}$, any $\sigma_{I\lambda\mu}$ is holomorphic on $W$ and $\int_W|\sigma_{I\lambda}|^2_{h\otimes\mathrm{det}\,h}dV_{\mathbb{C}^n}<+\infty$ where $\sigma_{I\lambda}(z)=\sum_\mu \sigma_{I\lambda\mu}(z)e_\mu\otimes e_{det}\in H^0(W,E\otimes\mathrm{det}\,E)$ by the assumption.
    Then the $L^2$-norm of $\sigma$ is the following:
    \begin{align*}
        ||\sigma||^2_{I^F_{\alpha,\tau}\otimes h\otimes\mathrm{det}\,h,\omega_P}\Big|_{W^*_r}=\sum_{I,\lambda}\int_{W^*_r}|\sigma_{I\lambda}(z)|^2_{h\otimes\mathrm{det}\,h}\prod^t_{j=1}|z_j|^{2(\tau_j-1)}(\log|z_j|^2)^{2(\alpha+\sum^b_{\nu=1}\delta_{ji_{p\nu}}-1)}dV_{\mathbb{C}^n}.
    \end{align*}

    Locally, we see that $h=\mathrm{det}\,h\cdot\widehat{h^*}$, where $\widehat{h^*}$ is the adjugate matrix of $h^*$. Since Griffiths semi-negativity of $h^*$, each element $\widehat{h^*_{kl}}$ of $\widehat{h^*}=(\widehat{h^*_{kl}})_{1\leq k,l\leq \mathrm{rank}\,E}$ is locally bounded [PT18,\,Lemma\,2.2.4].
    We show that for any $I$ and $\lambda$, 
    \begin{align*}
        J_{I,\lambda}:=&\int_{W^*_r}|\sigma_{I\lambda}(z)|^2_{h\otimes\mathrm{det}\,h}\prod^t_{j=1}|z_j|^{2(\tau_j-1)}(\log|z_j|^2)^{2(\alpha+\sum^b_{\nu=1}\delta_{ji_{p\nu}}-1)}dV_{\mathbb{C}^n}\\
        =&\int_{W^*_r}\Bigl(\sum_{k,l}\overline{\sigma}_{I\lambda k}(z)\widehat{h^*_{kl}}\sigma_{I\lambda l}(z)\Bigr)(\mathrm{det}\,h)^2\prod^t_{j=1}|z_j|^{2(\tau_j-1)}(\log|z_j|^2)^{2(\alpha+\sum^b_{\nu=1}\delta_{ji_{p\nu}}-1)}dV_{\mathbb{C}^n}<+\infty.
    \end{align*}

    Here, $\sum_{k,l}\overline{\sigma}_{I\lambda k}(z)\widehat{h^*_{kl}}\sigma_{I\lambda l}(z)\geq0$.
    Let $f_p:=(\sum_{k,l}\overline{\sigma}_{I\lambda k}(z)\widehat{h^*_{kl}}\sigma_{I\lambda l}(z))^{1/p}(\mathrm{det}\,h)^2$ and $g:=(\sum_{k,l}\overline{\sigma}_{I\lambda k}(z)\widehat{h^*_{kl}}\sigma_{I\lambda l}(z))^{1-1/p}\prod^t_{j=1}|z_j|^{2(\tau_j-1)}(-\log|z_j|^2)^{2(\alpha+\sum^b_{\nu=1}\delta_{ji_{p\nu}}-1)}$.
    By the H\"older's inequality, for any $p>1$ we get 
    \begin{align*}
        J_{I,\lambda}=\int_{W^*_r}|f_pg|dV_{\mathbb{C}^n}\leq\Bigl(\int_{W^*_r}|f_p|^pdV_{\mathbb{C}^n}\Bigr)^{1/p}\cdot\Bigl(\int_{W^*_r}|g|^qdV_{\mathbb{C}^n}\Bigr)^{1/q},
    \end{align*}
    where $1=1/p+1/q$ and get
    \begin{align*}
        \int_{W^*_r}|f_p|^pdV_{\mathbb{C}^n}=\int_{W^*_r}\Bigl(\sum_{k,l}\overline{\sigma}_{I\lambda k}(z)\widehat{h^*_{kl}}\sigma_{I\lambda l}(z)\Bigr)(\mathrm{det}\,h)^{2p}dV_{\mathbb{C}^n}
        =\int_{W^*_r}|\sigma_{I\lambda}|^2_{h\otimes(\mathrm{det}\,h)^{1+2(p-1)}}dV_{\mathbb{C}^n}
    \end{align*}

    From $\sigma_{I\lambda}\in\mathscr{E}(h\otimes\mathrm{det}\,h)(W)$ and the strong openness property on vector bundles (see \cite{LYZ21}) and Nakano semi-positivity of $h\otimes\mathrm{det}\,h$, for some $r'\in(0,r)$ there exists $\varepsilon>0$ such that $\int_{W_{r'}}|\sigma_{I\lambda}|^2_{h\otimes(\mathrm{det}\,h)^{1+\varepsilon}}dV_{\mathbb{C}^n}<+\infty$. 
    Then we consider the integral of $g$.
    \begin{align*}
        \int_{W^*_r}|g|^qdV_{\mathbb{C}^n}
        &=\int_{W^*_r}\Bigl(\sum_{k,l}\overline{\sigma}_{I\lambda k}(z)\widehat{h^*_{kl}}\sigma_{I\lambda l}(z)\Bigr)\prod^t_{j=1}|z_j|^{2q(\tau_j-1)}(-\log|z_j|^2)^{2q(\alpha+\sum^b_{\nu=1}\delta_{ji_{p\nu}}-1)}dV_{\mathbb{C}^n}\\
        &\leq C(\pi r^2)^{n-t+1}\prod^t_{j=1}\int_{|z_j|<r}|z_j|^{2q(\tau_j-1)}(-\log|z_j|^2)^{2q(\alpha+\sum^b_{\nu=1}\delta_{ji_{p\nu}}-1)}dz_j\wedge d\overline{z}_j\\
        &=C(\pi r^2)^{n-t+1}\prod^t_{j=1}2\pi\int^r_0r_j^{2q(\tau_j-1)+1}(-\log r_j^2)^{2q(\alpha+\sum^b_{\nu=1}\delta_{ji_{p\nu}}-1)}dr_j,
    \end{align*}
    where $C=\sup_{W_r}\sum_{k,l}\overline{\sigma}_{I\lambda k}(z)\widehat{h^*_{kl}}\sigma_{I\lambda l}(z)<+\infty$. 
    In fact, by Griffiths semi-positivity of $h$, the singular Hermitian metric $\widehat{h^*}=h/\mathrm{det}\,h$ is Griffiths semi-negative (see \cite{LYZ21}) and, the norm $\sum_{k,l}\overline{\sigma}_{I\lambda k}(z)\widehat{h^*_{kl}}\sigma_{I\lambda l}(z)=||\sigma_{I\lambda}||^2_{\frac{h}{\mathrm{det}\,h}}$ is bounded by plurisubharmonicity.
    Hence, by Example \ref{Example of integral} if $\tau_j>1-1/q=1/p$ then $\int_{W^*_r}|g|^qdV_{\mathbb{C}^n}<+\infty$. And if $\varepsilon=2(p-1)>0$ and $\tau_j>1/p=2/(2+\varepsilon)$ then $\sigma$ is $L^2$-integrable on $W^*_{r'}$.
    This concludes the case without the Lelong number condition, i.e. the case of $\tau_j=1$.

    Finally, we show the case assuming $\nu(-\log\mathrm{det}\,h,x)<\delta$ for all points in $D$.
    Then $e^{\frac{2}{\delta}\log\mathrm{det}\,h}=(\mathrm{det}\,h)^{2/\delta}$ is locally integrable around $x_0\in D$ by Skoda's result.
    Let $g\!:=\!(\sum_{k,l}\overline{\sigma}_{I\lambda k}(z)\widehat{h^*_{kl}}\sigma_{I\lambda l}(z))\prod^t_{j}|z_j|^{2(\tau_j-1)}(-\log|z_j|^2)^{2(\alpha+\sum^b_{\nu}\delta_{ji_{p\nu}}-1)}$.
    By H\"older's inequality, 
    \begin{align*}
        J_{I,\lambda}=\int_{W^*_r}|g|(\mathrm{det}\,h)^2dV_{\mathbb{C}^n}&\leq\Bigl(\int_{W^*_r} (\mathrm{det}\,h)^{2/\delta}dV_{\mathbb{C}^n}\Bigr)^{\delta}\cdot\Bigl(\int_{W^*_r}|g|^qdV_{\mathbb{C}^n}\Bigr)^{1/q},
    \end{align*}
    where $q=1/(1-\delta)$. Similarly to the above,
    \begin{align*}
        \int_{W^*_r}|g|^qdV_{\mathbb{C}^n}\leq C(\pi r^2)^{n-t+1}\prod^t_{j=1}2\pi\int^r_0r_j^{2q(\tau_j-1)+1}(\log r_j^2)^{2q(\alpha+\sum^b_{\nu=1}\delta_{ji_{p\nu}}-1)}dr_j,
    \end{align*}
    where $C:=\sup_{W^r}(\sum_{k,l}\overline{\sigma}_{I\lambda k}(z)\widehat{h^*_{kl}}\sigma_{I\lambda l}(z))^{2q}<+\infty$.
    Hence, by Example \ref{Example of integral} if $\tau_j>1-1/q=\delta$ then $\int_{W^*_r}|g|^qdV_{\mathbb{C}^n}<+\infty$ and $\sigma$ is $L^2$-integrable on $W^*_{r}$. 
\end{proof}

\vspace{2mm}

$\mathit{Proof ~of ~Theorem ~\ref{L2-type Log Dolbeault isomorphism with sHm vector}}$.
    From Proposition \ref{L2-type Log Dol isom at q=0 with sHm vector}, it is sufficient to show the exactness of the $L^2$-Dolbeault complex $(\mathscr{L}^{p,\ast}_{F\otimes E\otimes\mathrm{det}\,E,h^F_Y\otimes h\otimes\mathrm{det}\,h,\omega_P},\overline{\partial})$ on $D$ at $q\geq1$. 
    Here, $h\otimes\mathrm{det}\,h$ is $L^2$-type Nakano semi-positive (see [Ina22,\,Theorem\,1.3]) by Griffiths semi-positivity of $h$.
    By using [Wat22b,\,Theorem\,4.8], the exactness of the $L^2$-Dolbeault complex $(\mathscr{L}^{p,\ast}_{F\otimes E\otimes\mathrm{det}\,E,h^F_Y\otimes h\otimes\mathrm{det}\,h,\omega_P},\overline{\partial})$ is shown similarly to the proof of Theorem \ref{L2-type Log Dolbeault isomorphism with sHm line}. \qed

\section{Logarithmic vanishing theorems}

\subsection{Logarithmic vanishing theorems involving multiplier ideal sheaves}

In this subsection, we prove Theorems \ref{Log V-thm of k-posi + psef no nu condition} and \ref{Log V-thm of k-posi + psef} and obtain logarithmic vanishing theorems for big line bundles.

\begin{theorem}\label{Log V-thm of k-posi + psef in 4}$(=\mathrm{Theorem}\,\ref{Log V-thm of k-posi + psef})$
    Let $X$ be a compact \kah manifold and $D=\sum^s_{j=1}D_j$ be a simple normal crossing divisor in $X$. 
    Let $A$ be a holomorphic line bundle and $L$ be a holomorphic line bundle equipped with a singular Hermitian metric $h$ which is pseudo-effective, i.e. $i\Theta_{L,h}\geq0$ in the sense of currents. 
    If there exists $0<\delta\leq1$ such that $\nu(-\log h,x)<2\delta$ for all points in $D$ and that for an $\mathbb{R}$-divisor $\Delta=\sum^s_{j=1}a_jD_j$ with $a_j\in(\delta,1]$, the $\mathbb{R}$-line bundle $A\otimes\mathcal{O}_X(\Delta)$ is $k$-positive. 
    Then for any nef line bundle $N$ and any $p,q\geq k$, we have that 
    \begin{align*}
        H^q(X,K_X\otimes\mathcal{O}_X(D)\otimes A\otimes N\otimes L\otimes \mathscr{I}(h))&=0,\\
        H^n(X,\Omega_X^p(\log D)\otimes A\otimes N\otimes L\otimes \mathscr{I}(h))&=0.
    \end{align*}
\end{theorem}

\begin{proof}
    Let $\omega$ be a fixed \kah metric on $X$. Let $F=A\otimes \mathcal{O}_X(\Delta)$. Since $F$ is a $k$-positive $\mathbb{R}$-line bundle, there exist smooth Hermitian metrics $h_A$ and $h_j$ on $A$ and $\mathcal{O}_X(D_j)$ respectively,
    such that the curvature form of the induced metric $h_F=h_A\prod^s_{j=1}h_j^{a_j}$ on $F$, i.e. $i\Theta_{F,h_F}=i\Theta_{A,h_A}+i\sum^s_{j=1}a_j\Theta_{\mathcal{O}_X(D_j),h_j}$,
    is semipositive and has at least $n-k+1$ positive eigenvalues at each point of $X$.

    Let $\lambda^1_{\omega,h_F}\leq\cdots\leq\lambda^n_{\omega,h_F}$ be the eigenvalues of $i\Theta_{F,h_F}$ with respect to $\omega$.
    Thus for any $j\geq k$ we have $\lambda^j_{\omega,h_F}\geq\lambda^k_{\omega,h_F}\geq\min_{x\in X}\lambda^k_{\omega,h_F}(x)=:c_0>0$.
    Without loss of generality, we assume $\gamma\in(0,1)$. Since $N$ is nef, there exists a smooth Hermitian metric $h_{N,\gamma}$ on $N$ such that $i\Theta_{N,h_{N,\gamma}}=-\idd\log h_{N,\gamma}>-\gamma\omega$.
    Let $\sigma_j$ be the defining section of $D_j$. Fix smooth Hermitian metrics $h_{D_j}:=||\bullet||^2_{D_j}$ on line bundles $\mathcal{O}_X(D_j)$.  
    We define the smooth Hermitian metric $h_\Delta:=\prod^s_{j=1}h^{a_j}_{D_j}$ on $\mathcal{O}_X(\Delta)$.

    The induced smooth Hermitian metric on $\mathscr{F}:=A\otimes N$ over $Y$ is defined by 
    \begin{align*}
        h^{\mathscr{F}}_{\alpha,\varepsilon,\tau}:=h_{N,\gamma}\cdot h_F\cdot (h_\Delta)^{-1}\cdot \prod^s_{j=1}||\sigma_j||^{2\tau_j}_{D_j}(\log(\varepsilon||\sigma_j||^2_{D_j}))^{2\alpha}.
    \end{align*}
    Here the constant $\alpha>0$ is chosen to be large enough and the constants $\tau_j\in(\delta,1]$ and $\varepsilon\in(0,1]$ are to be determined later. 
    A straightforward computation shows that 
    \begin{align*}
        i\Theta_{\mathscr{F},h^{\mathscr{F}}_{\alpha,\varepsilon,\tau}}&=i\Theta_{F,h^F}+i\Theta_{N,h_{N,\gamma}}+i\sum^s_{j=1}(\tau_j-a_j)\Theta_{\mathcal{O}_X(D_j),h_{D_j}}\\
        &+i\sum^s_{j=1}\frac{2\alpha \Theta_{\mathcal{O}_X(D_j),h_{D_j}}}{\log(\varepsilon||\sigma_j||^2_{D_j})}+i\sum^s_{j=1}\frac{2\alpha\partial\log||\sigma_j||^2_{D_j}\wedge\overline{\partial}\log||\sigma_j||^2_{D_j}}{(\log(\varepsilon||\sigma_j||^2_{D_j}))^2}.
    \end{align*}

    Since $a_j\in(\delta,1]$, for a fixed large $\alpha$, we can choose $\tau_1,\cdots,\tau_s\in(\delta,1]$ and $\varepsilon$ such that $\tau_j-a_j,\varepsilon$ are small enough and 
    \begin{align*}
        -\frac{\gamma}{2}\omega\leq i\sum^s_{j=1}(\tau_j-a_j)\Theta_{\mathcal{O}_X(D_j),h_{D_j}}\leq \frac{\gamma}{2}\omega,\quad -\frac{\gamma}{2}\omega\leq i\sum^s_{j=1}\frac{2\alpha \Theta_{\mathcal{O}_X(D_j),h_{D_j}}}{\log(\varepsilon||\sigma_j||^2_{D_j})} \leq \frac{\gamma}{2}\omega.
    \end{align*}
    Note that the constants $\tau_j$ and $\varepsilon$ are thus fixed, and the choice of $\varepsilon$ depends on $\alpha$. Let
    \begin{align*}
        \omega_Y=i\Theta_{\mathscr{F},h^{\mathscr{F}}_{\alpha,\varepsilon,\tau}}+2(4n+1)\gamma\omega.
    \end{align*}
    It is easy to check that $\omega_Y$ is a Poincar\'e type \kah metric on $Y$. Then we have that 
    \begin{align*}
        i\Theta_{\mathscr{F},h^{\mathscr{F}}_{\alpha,\varepsilon,\tau}}\geq i\Theta_{F,h^F}-2\gamma\omega, \quad \omega_Y\geq8n\gamma\omega
    \end{align*}
    on $Y$. This implies that $i\Theta_{\mathscr{F},h^{\mathscr{F}}_{\alpha,\varepsilon,\tau}}=\omega_Y-2(4n+1)\gamma\omega\geq-\frac{1}{4n}\omega_Y$.
    By exactly the same argument as in the proof of Theorem \ref{L2-type Log Dolbeault isomorphism with sHm line}, when $\alpha$ is large enough, we obtain 
    \begin{align*}
        H^q(X,\Omega^p_X(\log D)\otimes\mathcal{O}_X(A\otimes N\otimes L)\otimes\mathscr{I}(h))\cong H^q(\Gamma(X,\mathscr{L}^{p,\ast}_{\mathscr{F}\otimes L,h^{\mathscr{F}}_{\alpha,\varepsilon,\tau}\otimes h,\omega_Y})).
    \end{align*}

    Finally, we prove the vanishing of $L^2$-cohomology groups using $L^2$-estimates equipped with a singular Hermitian metric.
    We set $\gamma=\frac{c_0}{32n^2}$. 
    Let $\lambda^1_{\omega,h^{\mathscr{F}}_{\alpha,\varepsilon,\tau}}\leq\cdots\leq\lambda^n_{\omega,h^{\mathscr{F}}_{\alpha,\varepsilon,\tau}}$ be the eigenvalues of $i\Theta_{\mathscr{F},h^{\mathscr{F}}_{\alpha,\varepsilon,\tau}}$ with respect to $\omega$.
    For any fixed point $x_0\in Y$, we can write that $\omega=i\sum^n_{j=1}\zeta_j\wedge \overline{\zeta}_j$ and $i\Theta_{\mathscr{F},h^{\mathscr{F}}_{\alpha,\varepsilon,\tau}}=i\sum^n_{j=1}\lambda^j_{\omega,h^{\mathscr{F}}_{\alpha,\varepsilon,\tau}} \zeta_j\wedge \overline{\zeta}_j$.
    Then we get 
    \begin{align*}
        i\Theta_{\mathscr{F},h^{\mathscr{F}}_{\alpha,\varepsilon,\tau}}
        =i\sum^n_{j=1}\frac{16n^2\lambda^j_{\omega,h^{\mathscr{F}}_{\alpha,\varepsilon,\tau}}}{16n^2\lambda^j_{\omega,h^N_{\alpha,\varepsilon,\tau}}+(4n+1)c_0} \zeta'_j\wedge \overline{\zeta}'_j
    \end{align*}
    where $\zeta'_j=\sqrt{\lambda^j_{\omega,h^{\mathscr{F}}_{\alpha,\varepsilon,\tau}}+2(4n+1)\gamma}\cdot\zeta_j$.
    Note that $\omega_Y=i\sum^n_{j=1}\zeta'_j\wedge \overline{\zeta}'_j$ and the eigenvalues of $i\Theta_{\mathscr{F},h^{\mathscr{F}}_{\alpha,\varepsilon,\tau}}$ with respect to $\omega_Y$ are 
    \begin{align*}
        \kappa_j:=\frac{16n^2\lambda^j_{\omega,h^{\mathscr{F}}_{\alpha,\varepsilon,\tau}}}{16n^2\lambda^j_{\omega,h^{\mathscr{F}}_{\alpha,\varepsilon,\tau}}+(4n+1)c_0}<1.
    \end{align*}
    Thus $\kappa_j\in[-1/4n,1)$ and we obtain $\lambda^j_{\omega,h^{\mathscr{F}}_{\alpha,\varepsilon,\tau}}\geq\lambda^j_{\omega,h_F}-2\gamma$.
    Hence for any $j\geq k$, we have that $\lambda^j_{\omega,h^{\mathscr{F}}_{\alpha,\varepsilon,\tau}}\geq\min_{x\in X}\lambda^k_{\omega,h_F}(x)-2\gamma=c_0-2\gamma=(1-\frac{1}{16n^2})c_0>0$ and
    \begin{align*}
        \kappa_j&
        =1-\frac{(4n+1)c_0}{16n^2\lambda^j_{\omega,h^{\mathscr{F}}_{\alpha,\varepsilon,\tau}}+(4n+1)c_0}
        \geq 1-\frac{(4n+1)c_0}{16n^2(1-1/16n^2)c_0+(4n+1)c_0}
        =1-\frac{1}{4n}.
    \end{align*}
    
    For any $p+q\geq n+k$ and any $u\in L^2_{p,q}(Y,\mathscr{F},h^{\mathscr{F}}_{\alpha,\varepsilon,\tau},\omega_Y)$, we obtain 
    \begin{align*}
        \langle[i\Theta_{\mathscr{F},h^{\mathscr{F}}_{\alpha,\varepsilon,\tau}},\Lambda_{\omega_Y}]u,u\rangle_{\omega_Y}&\geq\biggl(\sum^q_{j=1}\kappa_j-\sum^n_{j=p+1}\kappa_j \biggr)|u|^2_{\omega_Y}\\
        &\geq\bigl((q-k)\bigl(1-\frac{1}{4n}\bigr)-\frac{k}{4n}-(n-p)\bigr)|u|^2_{\omega_Y}\\
        &=\bigl((p+q-n-k)-\frac{q-k}{4n}-\frac{k}{4n}\bigr)|u|^2_{\omega_Y}\\
        &\geq\frac{1}{2}|u|^2_{\omega_Y}.
    \end{align*}
    Hence, $A^{p,q}_{\mathscr{F},h^{\mathscr{F}}_{\alpha,\varepsilon,\tau},\omega_Y}\geq1/2$ for $p+q\geq n+k$ and we get 
    \begin{align*}
        B_{h^{\mathscr{F}}_{\alpha,\varepsilon,\tau},\omega_Y}:=[i\Theta_{\mathscr{F},h^{\mathscr{F}}_{\alpha,\varepsilon,\tau}}\otimes\mathrm{id}_L,\Lambda_{\omega_Y}]\geq\frac{1}{2}.
    \end{align*}

    By compact-ness of $X$, for any $\overline{\partial}$-closed $f\in \Gamma(X,\mathscr{L}^{n,q}_{\mathscr{F}\otimes L,h^{\mathscr{F}}_{\alpha,\varepsilon,\tau}\otimes h,\omega_Y})$, we get 
    \begin{align*}
        \int_{Y}\langle B_{h^{\mathscr{F}}_{\alpha,\varepsilon,\tau},\omega_Y}^{-1}f,f\rangle_{h^{\mathscr{F}}_{\alpha,\varepsilon,\tau}\otimes h,\omega_Y}dV_{\omega_Y}\leq2\int_{Y}|f|^2_{h^{\mathscr{F}}_{\alpha,\varepsilon,\tau}\otimes h,\omega_Y}dV_{\omega_Y}<+\infty,
    \end{align*}
    i.e. $f\in L^2_{n,q}(Y,\mathscr{F}\otimes L,h^{\mathscr{F}}_{\alpha,\varepsilon,\tau}\otimes h,\omega_Y)$.
    Here, $h^{\mathscr{F}}_{\alpha,\varepsilon,\tau}$ is a smooth on $Y$.
    Since Theorem \ref{L2-estimate of k-posi + psef}, there exists $u\in L^2_{n,q-1}(Y,\mathscr{F}\otimes L,h^{\mathscr{F}}_{\alpha,\varepsilon,\tau}\otimes h,\omega_Y)$ such that $\overline{\partial}u=f$ on $Y$ and 
    \begin{align*}
        \int_{Y}|u|^2_{h^{\mathscr{F}}_{\alpha,\varepsilon,\tau}\otimes h,\omega_Y}dV_{\omega_Y}\leq\int_{Y}\langle B_{h^{\mathscr{F}}_{\alpha,\varepsilon,\tau},\omega_Y}^{-1}f,f\rangle_{h^{\mathscr{F}}_{\alpha,\varepsilon,\tau}\otimes h,\omega_Y}dV_{\omega_Y}<+\infty,
    \end{align*}
    where $|u|^2_{h^{\mathscr{F}}_{\alpha,\varepsilon,\tau}\otimes h,\omega_Y}$ is locally integrable on $Y$.
    From the following lemma, letting $u=0$ on $D$ then we have that $\overline{\partial}u=f$ on $X$ and $u\in \Gamma(X,\mathscr{L}^{n,q-1}_{\mathscr{F}\otimes L,h^{\mathscr{F}}_{\alpha,\varepsilon,\tau}\otimes h,\omega_Y})$. 
    
    Hence, we obtain $H^q(X,K_X(\log D)\otimes\mathcal{O}_X(A\otimes N\otimes L)\otimes\mathscr{I}(h))=0$ for $q\geq k$.

    Similarly proved in the case of $(p,n)$-forms. 
\end{proof}

\begin{lemma}\label{Ext d-equation for hypersurface}$\mathrm{(cf.\,[Dem82,\,Lemma\,6.9])}$ 
    Let $\Omega$ be an open subset of $\mathbb{C}^n$ and $Z$ be a complex analytic subset of $\Omega$. Assume that $u$ is a $(p,q-1)$-form with $L^2_{loc}$ coefficients and $g$ is a $(p,q)$-form with $L^1_{loc}$ coefficients such that $\overline{\partial}u=g$ on $\Omega\setminus Z$ (in the sense of currents).
    Then $\overline{\partial}u=g$ on $\Omega$.
\end{lemma}

Theorem \ref{Log V-thm of k-posi + psef no nu condition} is shown similarly to the proof of Theorem \ref{Log V-thm of k-posi + psef in 4} as the case $\tau_j=1$ by using Theorem \ref{L2-type Log Dolbeault isomorphism with sHm line}.
We obtain the following logarithmic vanishing for big line bundles which is an analogue of Theorem \ref{Log V-thm of k-posi + psef}.

\begin{corollary}\label{Log V-thm for big in 4}
    Let $X$ be a compact \kah manifold and $D=\sum^s_{j=1}D_j$ be a simple normal crossing divisor in $X$. 
    Let $A$ be a holomorphic line bundle and $L$ be a holomorphic line bundle equipped with a singular Hermitian metric $h$ which is big. 
    If there exists $0<\delta\leq1$ such that $\nu(-\log h,x)<2\delta$ for all points in $D$ and that for an $\mathbb{R}$-divisor $\Delta=\sum^s_{j=1}a_jD_j$ with $a_j\in(\delta,1]$, the $\mathbb{R}$-line bundle $A\otimes\mathcal{O}_X(\Delta)$ is semi-positive. 
    Then we have the following
    \begin{align*}
        H^q(X,K_X\otimes\mathcal{O}_X(D)\otimes A\otimes L\otimes \mathscr{I}(h))&=0,\\
        H^n(X,\Omega_X^p(\log D)\otimes A\otimes L\otimes \mathscr{I}(h))&=0
    \end{align*}
    for $p,q\geq 1$.
\end{corollary}

\begin{proof}
    Let $\omega$ be a fixed \kah metric on $X$. Let $F=A\otimes \mathcal{O}_X(\Delta)$. Since $F$ is a $k$-positive $\mathbb{R}$-line bundle, there exist smooth Hermitian metrics $h_A$ and $h_j$ on $A$ and $\mathcal{O}_X(D_j)$ respectively,
    such that the curvature form of the induced metric $h_F=h_A\prod^s_{j=1}h_j^{a_j}$ on $F$, i.e. $i\Theta_{F,h_F}=i\Theta_{A,h_A}+i\sum^s_{j=1}a_j\Theta_{\mathcal{O}_X(D_j),h_j}$, is semipositive.

    Since $h$ is big, there exists $\gamma>0$ such that $i\Theta_{L,h}\geq3\gamma\omega$ in the sense of currents. 
    Let $\sigma_j$ be the defining section of $D_j$. Fix smooth Hermitian metrics $h_{D_j}:=||\bullet||^2_{D_j}$ on line bundles $\mathcal{O}_X(D_j)$. 
    We define the smooth Hermitian metric $h_\Delta:=\prod^s_{j=1}h^{a_j}_{D_j}$ on $\mathcal{O}_X(\Delta)$.

    The induced smooth Hermitian metric on $A$ over $Y$ is defined by 
    \begin{align*}
        h^{A}_{\alpha,\varepsilon,\tau}:=h_F\cdot (h_\Delta)^{-1}\cdot \prod^s_{j=1}||\sigma_j||^{2\tau_j}_{D_j}(\log(\varepsilon||\sigma_j||^2_{D_j}))^{2\alpha}.
    \end{align*}
    Here the constant $\alpha>0$ is chosen to be large enough and the constants $\tau_j\in(\delta,1]$ and $\varepsilon\in(0,1]$ are to be determined later. 
    A straightforward computation shows that 
    \begin{align*}
        i\Theta_{A,h^{A}_{\alpha,\varepsilon,\tau}}&=i\Theta_{F,h_F}+i\sum^s_{j=1}(\tau_j-a_j)\Theta_{\mathcal{O}_X(D_j),h_{D_j}}\\
        &+i\sum^s_{j=1}\frac{2\alpha \Theta_{\mathcal{O}_X(D_j),h_{D_j}}}{\log(\varepsilon||\sigma_j||^2_{D_j})}+i\sum^s_{j=1}\frac{2\alpha\partial\log||\sigma_j||^2_{D_j}\wedge\overline{\partial}\log||\sigma_j||^2_{D_j}}{(\log(\varepsilon||\sigma_j||^2_{D_j}))^2}.
    \end{align*}

    Since $a_j\in(\delta,1]$, for a fixed large $\alpha$, we can choose $\tau_1,\cdots,\tau_s\in(\delta,1]$ and $\varepsilon$ such that $\tau_j-a_j,\varepsilon$ are small enough and 
    \begin{align*}
        -\frac{\gamma}{2}\omega\leq i\sum^s_{j=1}(\tau_j-a_j)\Theta_{\mathcal{O}_X(D_j),h_{D_j}}\leq \frac{\gamma}{2}\omega,\quad -\frac{\gamma}{2}\omega\leq i\sum^s_{j=1}\frac{2\alpha \Theta_{\mathcal{O}_X(D_j),h_{D_j}}}{\log(\varepsilon||\sigma_j||^2_{D_j})} \leq \frac{\gamma}{2}\omega.
    \end{align*}
    Note that the constants $\tau_j$ and $\varepsilon$ are thus fixed, and the choice of $\varepsilon$ depends on $\alpha$. 
    We have that $i\Theta_{A,h^{A}_{\alpha,\varepsilon,\tau}}\geq-\gamma\omega$ and $i\Theta_{A\otimes L,h^{A}_{\alpha,\varepsilon,\tau}\otimes h}\geq2\gamma\omega$. 
    We set 
    \begin{align*}
        \omega_Y=i\Theta_{A,h^{A}_{\alpha,\varepsilon,\tau}}+2\gamma\omega>0.
    \end{align*}
    It is easy to check that $\omega_Y$ is a Poincar\'e type \kah metric on $Y$. 
    By exactly the same argument as in the proof of Theorem \ref{L2-type Log Dolbeault isomorphism with sHm line}, when $\alpha$ is large enough, we obtain 
    \begin{align*}
        H^q(X,\Omega^p_X(\log D)\otimes\mathcal{O}_X(A\otimes L)\otimes\mathscr{I}(h))\cong H^q(\Gamma(X,\mathscr{L}^{p,\ast}_{A\otimes L,h^{A}_{\alpha,\varepsilon,\tau}\otimes h,\omega_Y})).
    \end{align*}

    Here, we get $i\Theta_{A\otimes L,h^{A}_{\alpha,\varepsilon,\tau}\otimes h}=i\Theta_{A,h^{A}_{\alpha,\varepsilon,\tau}}+i\Theta_{L,h}=\omega_Y-2\gamma\omega+i\Theta_{L,h}\geq\omega_Y+\gamma\omega>\omega_Y$
    in the sense of currents on $Y$, i.e. $h^{A}_{\alpha,\varepsilon,\tau}\otimes h$ is singular positive Hermitian metric on $Y$ which is uniformly bounded from below by $\omega_Y$.

    By Demailly's approximation to $h$ on $X$ (see \cite{Dem94}), there is a sequence of singular Hermitian metrics $(h_\nu)_{\nu\in\mathbb{N}}$ on $L$ over $X$ such that 
    \begin{itemize}
        \item [(i)] $h_\nu$ is smooth in the complement $X\setminus Z_\nu$ of an analytic set $Z_\nu\subset X$,
        \item [(ii)] $(h_\nu)_{\nu\in\mathbb{N}}$ is a increasing sequence and $h=\lim_{\nu\to+\infty}h_\nu$,
        \item [(iii)] $i\Theta_{L,h_\nu}\geq(3\gamma-\beta_\nu)\omega$ in the sense of currents, where $\lim_{\nu\to+\infty}\beta_\nu=0$.
    \end{itemize}
    Then for any enough large $\nu\in\mathbb{N}$, we get 
    \begin{align*}
        i\Theta_{A\otimes L,h^{A}_{\alpha,\varepsilon,\tau}\otimes h_\nu}
        =\omega_Y-2\gamma\omega+i\Theta_{L,h_\nu}\geq\omega_Y+(\gamma-\beta_\nu)\omega>\omega_Y
    \end{align*}
    in the sense of currents on $Y$, where this curvature inequality is smooth on $Y\setminus Z_\nu$.

    For any $\overline{\partial}$-closed $f\in \Gamma(X,\mathscr{L}^{n,q}_{A\otimes L,h^{A}_{\alpha,\varepsilon,\tau}\otimes h,\omega_Y})$, we get $\int_{Y}|f|^2_{h^{\mathscr{F}}_{\alpha,\varepsilon,\tau}\otimes h,\omega_Y}dV_{\omega_Y}<+\infty$ by compactness of $X$.
    Similarly to the proof of [Wat22b,\,Theorem\,4.6], there exists $u\in L^2_{n,q-1}(Y,A\otimes L,h^{A}_{\alpha,\varepsilon,\tau}\otimes h,\omega_Y)$ such that $\overline{\partial}u=f$ on $Y$ and 
    \begin{align*}
        \int_{Y}|u|^2_{h^{A}_{\alpha,\varepsilon,\tau}\otimes h,\omega_Y}dV_{\omega_Y}\leq\frac{2}{q}\int_{Y}|f|^2_{h^{A}_{\alpha,\varepsilon,\tau}\otimes h,\omega_Y}dV_{\omega_Y}.
    \end{align*}
    where $|u|^2_{h^{A}_{\alpha,\varepsilon,\tau}\otimes h,\omega_Y}$ is locally integrable on $Y$.
    From Lemma \ref{Ext d-equation for hypersurface}, 
    letting $u=0$ on $D$ then we have that $\overline{\partial}u=f$ on $X$ and $u\in\Gamma(X,\mathscr{L}^{n,q-1}_{\mathscr{F}\otimes L,h^{\mathscr{F}}_{\alpha,\varepsilon,\tau}\otimes h,\omega_Y})$. 
    
    Hence, we obtain $H^q(X,K_X(\log D)\otimes\mathcal{O}_X(A\otimes L)\otimes\mathscr{I}(h))=0$ for any $q\geq1$.

    Similarly proved in the case of $(p,n)$-forms.
\end{proof}

The following logarithmic vanishing theorem is shown similarly to the proof of Theorem \ref{Log V-thm for big in 4} as the case $\tau_j=1$ by using Theorem \ref{L2-type Log Dolbeault isomorphism with sHm line}.

\begin{theorem}\label{Log V-thm for big no nu condition}
    Let $X$ be a compact \kah manifold and $D$ be a simple normal crossing divisor in $X$. 
    Let $A$ be a holomorphic line bundle and $L$ be a holomorphic line bundle equipped with a singular Hermitian metric $h$ which is big. 
    If $A\otimes\mathcal{O}_X(D)$ is semi-positive then we have that 
    \begin{align*}
        H^q(X,K_X\otimes\mathcal{O}_X(D)\otimes A\otimes L\otimes \mathscr{I}(h))&=0,\\
        H^n(X,\Omega_X^p(\log D)\otimes A\otimes L\otimes \mathscr{I}(h))&=0
    \end{align*}
    for $p,q\geq 1$.
\end{theorem}

Theorem \ref{Log V-thm for big no nu condition} for $(n,q)$-forms is equivalent to Nadel vanishing by taking $A=\mathcal{O}_X(D)^*$.
Theorem \ref{Log V-thm for big no nu condition} for $(p,n)$-forms is a generalization of the following Bogomolov-Sommese logarithmic vanishing theorem for the case of big line bundles to that involving multiplier ideal sheaves.
In fact, if $L$ is big then $\kappa(L)=n$, and we can take $A$ as $\mathcal{O}_X(D)^*$.

\begin{theorem}$\mathrm{(Bogomolov}$-$\mathrm{Sommese~ logarithmic~ vanishing~ theorem,\,cf.}$ \cite{Bog78}$)$\label{Log V-thm for Bogomolove-Sommese type}
    Let $X$ be a projective manifold and $\mathcal{L}$ be an invertible sheaf on $X$.
    Then for any simple normal crossing divisor $D$, we have 
    \begin{align*}
        H^0(X,\Omega_X^a(\log D)\otimes \mathcal{L}^*)=0
    \end{align*}
    for $a<\kappa(\mathcal{L})$. In other words, we have 
    \begin{align*}
        H^n(X,\Omega_X^p(\log D)\otimes \mathcal{O}_X(D)^*\otimes \mathcal{L})=0
    \end{align*}
    for $p>n-\kappa(\mathcal{L})$.
\end{theorem}

Finally, we obtain the following logarithmic vanishing for nef and big line bundles which is the Kawamata-Viehweg type.
This cohomology vanishing for $(n,q)$-forms also follows from the Kawamata-Viehweg vanishing theorem.

\begin{corollary}\label{Log V-thm for nef big}
    Let $X$ be a compact \kah manifold, $D=\sum^s_{j=1}D_j$ be a simple normal crossing divisor in $X$ and $L$ be a nef and big line bundle. 
    If there exist some real positive constants $a_j>0$ such that $\sum^s_{j=1}a_jD_j$ is semi-positive $\mathbb{R}$-divisor then we have that 
    \begin{align*}
        H^q(X,K_X\otimes\mathcal{O}_X(D)\otimes L)&=0,\\
        H^n(X,\Omega_X^p(\log D)\otimes L)&=0
    \end{align*}
    for $p,q\geq 1$.
\end{corollary}

\begin{proof}
    We can set $A=\mathcal{O}_X$ and $\Delta=\frac{1}{1+\sum^s_{j=1}a_j}\sum^s_{j=1}a_jD_j$. Then $A\otimes\mathcal{O}_X(\Delta)=\mathcal{O}_X(\Delta)$ is semi-positive $\mathbb{R}$-line bundle.
    Let $\omega$ be a \kah metric on $X$ and $0<\delta:=\min_j\frac{a_j}{2+\sum^s_{j=1}a_j}<\frac{a_j}{1+\sum^s_{j=1}a_j}<1$. Since $L$ is nef and big, $L$ has a singular Hermitian metric $h$ such that $\max_{x\in X}\nu(-\log h,x)<\delta$ and $i\Theta_{L,h}\geq\varepsilon\omega$ for some $\varepsilon>0$ (see [Dem10,\,Corollary\,6.19]).
    Hence, the proofs is complete from Corollary \ref{Log V-thm for big in 4}.
\end{proof}

\subsection{Logarithmic vanishing theorems involving $L^2$-subsheaves}

In this subsection, we get logarithmic vanishing theorems for singular Hermitian metrics on holomorphic vector bundles with (semi)-positivity in the sense of Griffiths and $L^2$-type Nakano.

\begin{theorem}\label{Log V-thm of k-posi + Grif}
    Let $X$ be a projective manifold and $D=\sum^s_{j=1}D_j$ be a simple normal crossing divisor in $X$. 
    Let $A$ be a holomorphic line bundle and $E$ be a holomorphic vector bundle equipped with a singular Hermitian metric $h$ which is Griffiths semi-positive. 
    If there exists $0<\delta\leq1$ such that $\nu(-\log \mathrm{det}\,h,x)<\delta$ for all points in $D$ and that for an $\mathbb{R}$-divisor $\Delta=\sum^s_{j=1}a_jD_j$ with $a_j\in(\delta,1]$, the $\mathbb{R}$-line bundle $A\otimes\mathcal{O}_X(\Delta)$ is $k$-positive. 
    Then for any nef line bundle $N$, we have that 
    \begin{align*}
        H^q(X,K_X\otimes\mathcal{O}_X(D)\otimes A\otimes N\otimes \mathscr{E}(h\otimes\mathrm{det}\,h))&=0,\\
        H^n(X,\Omega_X^p(\log D)\otimes A\otimes N\otimes \mathscr{E}(h\otimes\mathrm{det}\,h))&=0
    \end{align*}
    for $p,q\geq k$.
\end{theorem}

\begin{proof}
    Let $\omega$ be a \kah metric on $X$ and $Y:=X\setminus D$. 
    Similar to the proof of Theorem \ref{Log V-thm of k-posi + psef in 4}, there exist a smooth Hermitian metric $h^{\mathscr{F}}_{\alpha,\varepsilon,\tau}$ on $\mathscr{F}:=A\otimes N$ over $Y$ and a Poincar\'e type \kah metric 
    $\omega_Y=i\Theta_{\mathscr{F},h^{\mathscr{F}}_{\alpha,\varepsilon,\tau}}+2(4n+1)\gamma\omega$ on $Y$ where $\gamma\in(0,1)$ and we get the following: 
    (I) An $L^2$-type Log Dolbeault isomorphism 
    \begin{align*}
        H^q(X,\Omega^p_X(\log D)\otimes A\otimes N\otimes\mathscr{E}(h\otimes\mathrm{det}\,h))\cong H^q(\Gamma(X,\mathscr{L}^{p,\ast}_{\mathscr{F}\otimes E\otimes\mathrm{det}\,E,h^{\mathscr{F}}_{\alpha,\varepsilon,\tau}\otimes h\otimes\mathrm{det}\,h})).
    \end{align*}
    (II) For any $(p,q)$-forms in $L^2_{p,q}(Y,\mathscr{F},h^{\mathscr{F}}_{\alpha,\varepsilon,\tau},\omega_Y)$ with $p+q\geq n+k$, we get 
    \begin{align*}
        B_{h^{\mathscr{F}}_{\alpha,\varepsilon,\tau},\omega_Y}=[i\Theta_{\mathscr{F},h^{\mathscr{F}}_{\alpha,\varepsilon,\tau}}\otimes\mathrm{id}_{E\otimes\mathrm{det}\,E},\Lambda_{\omega_Y}]\geq \frac{1}{2}.
    \end{align*}

    Here, $h\otimes\mathrm{det}\,h$ is $L^2$-type Nakano semi-positive (see [Ina22,\,Theorem\,1.3]) and dual Nakano semi-positive (see [Wat22b,\,Theorem\,1.5]) by Griffiths semi-positivity of $h$.
    By projectivity of $X$, there exists a hypersurface $H$ such that $X\setminus H$ is Stein. Thus, $Y_H:=Y\setminus H=(X\setminus H)\setminus D$ is also Stein.
    For any $\overline{\partial}$-closed $f\in\Gamma(X,\mathscr{L}^{n,q}_{\mathscr{F}\otimes E\otimes\mathrm{det}\,E,h^{\mathscr{F}}_{\alpha,\varepsilon,\tau}\otimes h\otimes\mathrm{det}\,h})$, 
    \begin{align*}
        \int_{Y_H}\langle B_{h^{\mathscr{F}}_{\alpha,\varepsilon,\tau},\omega_Y}^{-1}f,f\rangle_{h^{\mathscr{F}}_{\alpha,\varepsilon,\tau}\otimes h\otimes\mathrm{det}\,h,\omega_Y}dV_{\omega_Y}\leq2\int_{Y_H}|f|^2_{h^{\mathscr{F}}_{\alpha,\varepsilon,\tau}\otimes h\otimes\mathrm{det}\,h,\omega_Y}dV_{\omega_Y}<+\infty,
    \end{align*}
    i.e. $f\in L^2_{n,q}(Y_H,\mathscr{F}\otimes E\otimes\mathrm{det}\,E,h^{\mathscr{F}}_{\alpha,\varepsilon,\tau}\otimes h\otimes\mathrm{det}\,h,\omega_Y)$. 
    
    Since Theorem \ref{L2-estimate of k-posi + Grif}, for any $q\geq k$ there exists $u\in L^2_{n,q-1}(Y_H,\mathscr{F}\otimes E\otimes\mathrm{det}\,E,h^{\mathscr{F}}_{\alpha,\varepsilon,\tau}\otimes h\otimes\mathrm{det}\,h,\omega_Y)$ such that $\overline{\partial}u=f$ on $Y_H$ and 
    \begin{align*}
        \int_{Y_H}|u|^2_{h^{\mathscr{F}}_{\alpha,\varepsilon,\tau}\otimes h\otimes\mathrm{det}\,h,\omega_Y}dV_{\omega_Y}\leq\int_{Y_H}\langle B_{h^{\mathscr{F}}_{\alpha,\varepsilon,\tau},\omega_Y}^{-1}f,f\rangle_{h^{\mathscr{F}}_{\alpha,\varepsilon,\tau}\otimes h\otimes\mathrm{det}\,h,\omega_Y}dV_{\omega_Y}<+\infty,
    \end{align*}
    where $|u|^2_{h^{\mathscr{F}}_{\alpha,\varepsilon,\tau}\otimes h\otimes\mathrm{det}\,h,\omega_Y}$ is locally integrable on $Y_H$.
    From Lemma \ref{Ext d-equation for hypersurface}, letting $u=0$ on $H\cup D$ then we have that $u\in\Gamma(X,\mathscr{L}^{n,q-1}_{\mathscr{F}\otimes E\otimes\mathrm{det}\,E,h^{\mathscr{F}}_{\alpha,\varepsilon,\tau}\otimes h\otimes\mathrm{det}\,h,\omega_Y})$ and $\overline{\partial}u=f$ on $X$. 
    
    Hence, we obtain $H^q(X,K_X(\log D)\otimes A\otimes N\otimes\mathscr{E}(h\otimes\mathrm{det}\,h))=0$ for any $q\geq k$.

    Similarly proved in the case of $(p,n)$-forms by using $L^2$-estimates for dual Nakano semi-positivity (see [Wat22b,\,Theorem\,4.9]).
\end{proof}

Theorem \ref{Log V-thm of k-posi + Grif no nu condition} is shown similarly to the proof of Theorem \ref{Log V-thm of k-posi + Grif} as the case $\tau_j=1$ by using Theorem \ref{L2-type Log Dolbeault isomorphism with sHm vector}.
As the case of vector bundles analogous to big line bundles, we obtain the following logarithmic vanishing theorems for strictly Griffiths $\rho_\omega$-positivity.

\begin{corollary}\label{Log V-thm of semi-posi + s-Grif}
    Let $X$ be a projective manifold, $\omega$ be a \kah metric on $X$ and $D=\sum^s_{j=1}D_j$ be a simple normal crossing divisor in $X$. 
    Let $A$ be a holomorphic line bundle and $E$ be a holomorphic vector bundle equipped with a singular Hermitian metric $h$. 
    We assume that $h$ is strictly Griffiths $\rho_\omega$-positive and that there exists $0<\delta\leq1$ such that $\nu(-\log \mathrm{det}\,h,x)<\delta$ for all points in $D$ and that for an $\mathbb{R}$-divisor $\Delta=\sum^s_{j=1}a_jD_j$ with $a_j\in(\delta,1]$, the $\mathbb{R}$-line bundle $A\otimes\mathcal{O}_X(\Delta)$ is semi-positive. 
    Then we have that
    \begin{align*}
        H^q(X,K_X\otimes\mathcal{O}_X(D)\otimes A\otimes \mathscr{E}(h\otimes\mathrm{det}\,h))&=0,\\
        H^n(X,\Omega_X^p(\log D)\otimes A\otimes \mathscr{E}(h\otimes\mathrm{det}\,h))&=0
    \end{align*}
    for $p,q\geq 1$.
\end{corollary}

\begin{proof}
    Similar to the proof of Corollary \ref{Log V-thm for big in 4}, there exist a smooth Hermitian metric $h^{A}_{\alpha,\varepsilon,\tau}$ on $A$ over $Y$ such that $i\Theta_{A,h^{A}_{\alpha,\varepsilon,\tau}}\geq-\rho\omega$ on $Y$
    and a Poincar\'e type \kah metric $\omega_Y=i\Theta_{A,h^{A}_{\alpha,\varepsilon,\tau}}+2\rho\omega$ on $Y$, 
    and we get the following $L^2$-type Log Dolbeault isomorphism 
    \begin{align*}
        H^q(X,\Omega^p_X(\log D)\otimes A\otimes\mathscr{E}(h\otimes\mathrm{det}\,h))\cong H^q(\Gamma(X,\mathscr{L}^{p,\ast}_{A\otimes E\otimes\mathrm{det}\,E,h^{A}_{\alpha,\varepsilon,\tau}\otimes h\otimes\mathrm{det}\,h})).
    \end{align*}

    By strictly Griffiths $\rho_\omega$-positivity of $h$, the singular Hermitian metric $h\otimes\mathrm{det}\,h$ is $L^2$-type strictly Nakano $(r+1)\rho_\omega$-positive by Theorem \ref{h s-Gri then h * det h s-Nak} and strictly dual Nakano $(r+1)\rho_\omega$-positive by Theorem \ref{h s-Gri then h * det h s-dual Nak}, where $\mathrm{rank}\,E=r$.

    We show that $h\otimes\mathrm{det}\,h\otimes h^{A}_{\alpha,\varepsilon,\tau}$ is $L^2$-type strictly Nakano $1_{\omega_Y}$-positive and strictly dual Nakano $1_{\omega_Y}$-positive on $Y$.
    Here, for any open subset $U\subset Y$ and any \kah potential $\varphi$ of $\omega$ on $U$, $h\otimes\mathrm{det}\,he^{(r+1)\rho\varphi}$ is $L^2$-type Nakano semi-positive and dual Nakano semi-positive. 
    For any \kah potential $\psi$ of $\omega_Y$ on $U$, we define the smooth Hermitian metric $h_A:=h^{A}_{\alpha,\varepsilon,\tau}e^{\psi-(r+1)\rho\varphi}$ on $A$ over $U$ then $h_A$ is semi-positive.
    In fact, 
    \begin{align*}
        i\Theta_{A,h_A}=i\Theta_{A,h^{A}_{\alpha,\varepsilon,\tau}}-\omega_Y+(r+1)\rho\omega=(r-1)\rho\omega\geq0.
    \end{align*}
    From [Wat22b,\,Theorem\,3.18], the singular Hermitian metric $h\otimes\mathrm{det}\,h\otimes h^{A}_{\alpha,\varepsilon,\tau}e^\psi=h\otimes$ $\mathrm{det}\,he^{(r+1)\rho\varphi}\otimes h_{A}$ is $L^2$-type Nakano semi-positive and dual Nakano semi-positive on $U$\!.

    By projectivity of $X$, there exists a hypersurface $H$ such that $X\setminus H$ is Stein and that $\omega$ has a \kah potential $\varphi$ on $X\setminus H$. Thus, $Y_H:=Y\setminus H=(X\setminus H)\setminus D$ is also Stein.
    Here, there exists a hypersurface $Z$ such that $Y_{HZ}:=Y_H\setminus Z$ is also Stein and that $\omega_Y$ has a \kah potential $\psi$ on $Y_{HZ}$, i.e. $\idd\psi=\omega_Y$.
    
    In fact, $\omega_Y=i\Theta_{A,h^{A}_{\alpha,\varepsilon,\tau}}+2\rho\omega=i\Theta_{A,h^{A}_{\alpha,\varepsilon,\tau}e^{-2\rho\varphi}}$ on $Y_H$.
    We take a global holomorphic section $\sigma_A\in H^0(Y_H,A)$ by Steinness of $Y_H$ and put the hypersurface $Z:=\{z\in Y_H\mid\sigma_A(z)=0\}$. For a smooth Hermitian metric $h^{A}_{\alpha,\varepsilon,\tau}e^{-2\rho\varphi}$ on $A$ over $Y_H$, the norm $||\sigma_A||^2_{h^{A}_{\alpha,\varepsilon,\tau}e^{-2\rho\varphi}}$ is a smooth semi-positive function on $Y_H$.
    In particular, defining $\psi:=-\log ||\sigma_A||^2_{h^{A}_{\alpha,\varepsilon,\tau}e^{-2\rho\varphi}}=-\log ||\sigma_A||^2_{h^{A}_{\alpha,\varepsilon,\tau}}+2\rho\varphi$ yields that $\psi$ is a \kah potential of $\omega_Y$ on $Y_{HZ}$ from the following calculation. For any local open subset $U\subset Y_{HZ}$, 
    \begin{align*}
        \idd\psi&=-\idd\log ||\sigma_A||^2_{h^{A}_{\alpha,\varepsilon,\tau}}+2\rho\idd\varphi=-\idd\log|\sigma_A|^2-\idd\log h^{A}_{\alpha,\varepsilon,\tau}+2\rho\omega\\
        &=i\Theta_{A,h^{A}_{\alpha,\varepsilon,\tau}}+2\rho\omega=\omega_Y,
    \end{align*}
    on $U$ where $-\log|\sigma_A|$ is pluriharmonic on $U$.

    For any $\overline{\partial}$-closed $f\in\Gamma(X,\mathscr{L}^{n,q}_{A\otimes E\otimes\mathrm{det}\,E,h^{A}_{\alpha,\varepsilon,\tau}\otimes h\otimes\mathrm{det}\,h})$, we get 
    \begin{align*}
        \int_{Y_{HZ}}\langle B_{e^{-\psi},\omega_Y}^{-1}f,f\rangle_{h^{A}_{\alpha,\varepsilon,\tau}\otimes h\otimes\mathrm{det}\,h,\omega_Y}dV_{\omega_Y}=\frac{1}{q}\int_{Y_{HZ}}|f|^2_{h^{A}_{\alpha,\varepsilon,\tau}\otimes h\otimes\mathrm{det}\,h,\omega_Y}dV_{\omega_Y}<+\infty,
    \end{align*}
    i.e. $f\in L^2_{n,q}(Y_{HZ},A\otimes E\otimes\mathrm{det}\,E,h^{A}_{\alpha,\varepsilon,\tau}\otimes h\otimes\mathrm{det}\,h,\omega_Y)$ by compactness of X, where $B_{e^{-\psi},\omega_Y}:=[\idd\psi\otimes\mathrm{id}_{A\otimes E\otimes\mathrm{det}\,E},\omega_Y]=q$ for any $(n,q)$-forms.

    Since $L^2$-type Nakano semi-positivity of $h\otimes\mathrm{det}\,h\otimes h^{A}_{\alpha,\varepsilon,\tau}e^\psi$ on $Y_{HZ}$, for a smooth Hermitian metric $e^{-\psi}$ on a trivial line bundle and for any $q\geq 1$ there exists $u\in L^2_{n,q-1}(Y_{HZ},A\otimes E\otimes\mathrm{det}\,E,h^{A}_{\alpha,\varepsilon,\tau}\otimes h\otimes\mathrm{det}\,h,\omega_Y)$ such that $\overline{\partial}u=f$ on $Y_{HZ}$ and 
    \begin{align*}
        \int_{Y_{HZ}}|u|^2_{h^{A}_{\alpha,\varepsilon,\tau}\otimes h\otimes\mathrm{det}\,h,\omega_Y}dV_{\omega_Y}\leq\int_{Y_{HZ}}\langle B_{e^{-\psi},\omega_Y}^{-1}f,f\rangle_{h^{A}_{\alpha,\varepsilon,\tau}\otimes h\otimes\mathrm{det}\,h,\omega_Y}dV_{\omega_Y}<+\infty,
    \end{align*}
    where $|u|^2_{h^{A}_{\alpha,\varepsilon,\tau}\otimes h\otimes\mathrm{det}\,h,\omega_Y}$ is locally integrable on $Y_{HZ}$.
    From Lemma \ref{Ext d-equation for hypersurface}, letting $u=0$ on $Z\cup H\cup D$ then we have that $\overline{\partial}u=f$ on $X$ and $u\in\Gamma(X,\mathscr{L}^{n,q-1}_{\mathscr{F}\otimes E\otimes\mathrm{det}\,E,h^{\mathscr{F}}_{\alpha,\varepsilon,\tau}\otimes h\otimes\mathrm{det}\,h,\omega_Y})$.

    Hence, we obtain $H^q(X,K_X(\log D)\otimes A\otimes N\otimes\mathscr{E}(h\otimes\mathrm{det}\,h))=0$ for any $q\geq 1$.

    Similarly proved in the case of $(p,n)$-forms by using $L^2$-estimates for dual Nakano semi-positivity (see [Wat22b,\,Theorem\,4.9]).
\end{proof}

\begin{theorem}\label{Log V-thm of semi-posi + s-Grif no nu condition}
    Let $X$ be a projective manifold, $\omega$ be a \kah metric on $X$ and $D$ be a simple normal crossing divisor in $X$. 
    Let $A$ be a holomorphic line bundle and $E$ be a holomorphic vector bundle equipped with a singular Hermitian metric $h$. 
    We assume that $A\otimes\mathcal{O}_X(D)$ is semi-positive and that $h$ is strictly Griffiths $\rho_\omega$-positive. 
    Then for any $p,q\geq 1$, we have that 
    \begin{align*}
        H^q(X,K_X\otimes\mathcal{O}_X(D)\otimes A\otimes \mathscr{E}(h\otimes\mathrm{det}\,h))&=0,\\
        H^n(X,\Omega_X^p(\log D)\otimes A\otimes \mathscr{E}(h\otimes\mathrm{det}\,h))&=0.
    \end{align*}
\end{theorem}

Here, we get an $L^2$-type Dolbeault isomorphism for $L^2$-type Nakano semi-positivity in the same way as in Theorem \ref{L2-type Log Dolbeault isomorphism with sHm vector},
where the Lelong number condition is $2\delta$ from $\delta$ since $\otimes\,\mathrm{det}\,h$ is not necessary.
Similar to the proofs of Theorems \ref{Log V-thm of k-posi + Grif} and Corollary \ref{Log V-thm of semi-posi + s-Grif}, we obtain the following logarithmic vanishing theorems for $L^2$-type Nakano positivity. 

\begin{corollary}\label{Log V-thm of k-posi + Nak}
    Let $X$ be a projective manifold and $D=\sum^s_{j=1}D_j$ be a simple normal crossing divisor in $X$. 
    Let $A$ be a holomorphic line bundle and $E$ be a holomorphic vector bundle equipped with a singular Hermitian metric $h$ which is $L^2$-type Nakano semi-positive. 
    If there exists $0<\delta\leq1$ such that $\nu(-\log \mathrm{det}\,h,x)<2\delta$ for all points in $D$ and that for an $\mathbb{R}$-divisor $\Delta=\sum^s_{j=1}a_jD_j$ with $a_j\in(\delta,1]$, the $\mathbb{R}$-line bundle $A\otimes\mathcal{O}_X(\Delta)$ is $k$-positive. 
    Then for any nef line bundle $N$ and any $q\geq k$, we have that 
    \begin{align*}
        H^q(X,K_X\otimes\mathcal{O}_X(D)\otimes A\otimes N\otimes \mathscr{E}(h))=0.
    \end{align*}
\end{corollary}

\begin{corollary}\label{Log V-thm of semi-posi + s-Nak}
    Let $X$ be a projective manifold, $D=\sum^s_{j=1}D_j$ be a simple normal crossing divisor and $\omega$ be a \kah metric. 
    Let $A$ be a holomorphic line bundle and $E$ be a holomorphic vector bundle equipped with a singular Hermitian metric $h$. 
    We assume that $h$ is $L^2$-type strictly Nakano $\rho_\omega$-positive. 
    If there exists $0<\delta\leq1$ such that $\nu(-\log \mathrm{det}\,h,x)<2\delta$ for all points in $D$ and that for an $\mathbb{R}$-divisor $\Delta=\sum^s_{j=1}a_jD_j$ with $a_j\in(\delta,1]$, the $\mathbb{R}$-line bundle $A\otimes\mathcal{O}_X(\Delta)$ is semi-positive. 
    Then for any $q\geq1$, we get
    \begin{align*}
        H^q(X,K_X\otimes\mathcal{O}_X(D)\otimes A\otimes \mathscr{E}(h))=0.
    \end{align*}
\end{corollary}

\subsection{Remarks and a counterexample}

In this subsection, we give a counterexample to the extension of Theorem \ref{Log V-thm for big no nu condition} to the same bidegree $(p,q)$ with $p+q>n$ as Theorem \ref{Log V-thm in HLWY16}, 
and state the relationship about whether or not there is a degree of freedom in positivity, i.e. a condition on the Lelong number, of vanishing theorems.

Here, the following corollary follows immediately from Theorem \ref{Log V-thm in HLWY16}.

\begin{corollary}\label{Log V-thm for smooth N-A-K type}$(\mathrm{simple ~Kodaira}$-$\mathrm{Akizuki}$-$\mathrm{Nakano ~type ~logarithmic ~vanishing})$
    Let $X$ be a compact \kah manifold, $D$ be a simple normal crossing divisor in $X$ and $A$ be a holomorphic line bundle.
    If $\mathcal{O}_X(D)$ is semi-positive and $A$ is positive then we have that
    \begin{align*}
        H^q(X,\Omega_X^p(\log D)\otimes A)=0 \qquad \mathit{for~any} \quad p+q>n.
    \end{align*}
\end{corollary}

\begin{remark}
    Theorem \ref{Log V-thm for big no nu condition} and Corollary \ref{Log V-thm for nef big}, that is logarithmic vanishing theorem for big line bundle, cannot be extended to the same bidegree $(p,q)$ with $p+q>n$ as the Kodaira-Akizuki-Nakano type logarithmic vanishing theorem such as Corollary \ref{Log V-thm for smooth N-A-K type}.
\end{remark}

In fact, we obtain the following counterexample to the extension of the Kodaira-Akizuki-Nakano type to nef and big by using Ramanujam's counterexample.
    
\vspace{2mm}
    
$\mathbf{Counterexample.}$ 
    Let $X$ be a blown up of one point $x_0$ in $\mathbb{P}^n$ and $\pi:X\to\mathbb{P}^n$ be the natural morphism. Clearly the line bundle $\pi^*\mathcal{O}_{\mathbb{P}^n}(1)$ is nef and big.
    Let $D$ be a prime divisor in $X$ such that $\mathcal{O}_X(D)$ is semi-positive and that $\pi^*\mathcal{O}_{\mathbb{P}^n}(1)|_D$ is big on $D$.
    Then we have the following non-vanishing cohomology
    \begin{align*}
        H^{n-1}(X,\Omega_X^{n-1}(\log D)\otimes\pi^*\mathcal{O}_{\mathbb{P}^n}(1))\ne 0.
    \end{align*}

\vspace{2mm}




For example, let $\sigma\in H^0(\mathbb{P}^n,\mathcal{O}_{\mathbb{P}^n}(1))$ such that $x_0\notin D_{\mathbb{P}^n}:=div\,\sigma$ and $D:=\pi^*D_{\mathbb{P}^n}$ be a prime divisor, 
then we have that $\mathcal{O}_X(D)=\pi^*\mathcal{O}_{\mathbb{P}^n}(1)$ is semi-positive by holomorphicity of $\pi$ and positivity of $\mathcal{O}_{\mathbb{P}^n}(1)$ 
and that $\pi^*\mathcal{O}_{\mathbb{P}^n}(1)|_{D_{\mathbb{P}^n}}$ is ample by $\pi|_{X\setminus Exc}:X\setminus Exc\to \mathbb{P}^n\setminus \{x_0\}$ is a biholomorphism, where $Exc$ is the exceptional divisor.

\begin{proof}
    By the analytical characterization of nef and big (see [Dem10,\,Corollary\,6.19]), there exist a singular Hermitian metric $h_{\pi^*O(1)}$ on $\pi^*\mathcal{O}_{\mathbb{P}^n}(1)$ such that 
    $\mathscr{I}(h_{\pi^*O(1)})=\mathcal{O}_X$ and $i\Theta_{\pi^*\mathcal{O}_{\mathbb{P}^n}(1),h_{\pi^*O(1)}}\geq\varepsilon\omega$ in the sense of currents for some $\varepsilon>0$, where $\omega$ is a \kah metric on $X$.
    Here, we already know Ramanujan's counterexample (see \cite{Ram72},\,[Dem-book,\,ChapterVII]): $\quad H^p(X,\Omega_X^p\otimes\pi^*\mathcal{O}_{\mathbb{P}^n}(1))\ne0\quad$ for $\,\, 0\leq p\leq n-1$.

    The exact sequence $0\longrightarrow \Omega_X^p \longrightarrow \Omega_X^p(\log D) \longrightarrow \Omega_D^p \longrightarrow 0$
    twisted by $\pi^*\mathcal{O}_{\mathbb{P}^n}(1)$ yields the induced exact sequence
    \begin{align*}
        \cdots &\longrightarrow H^{n-2}(D,K_D\otimes \pi^*\mathcal{O}_{\mathbb{P}^n}(1)|_D) \longrightarrow H^{n-1}(X,\Omega_X^{n-1}\otimes \pi^*\mathcal{O}_{\mathbb{P}^n}(1)) \\
        &\longrightarrow H^{n-1}(X,\Omega_X^{n-1}(\log D)\otimes \pi^*\mathcal{O}_{\mathbb{P}^n}(1)) \longrightarrow H^{n-1}(D,K_D\otimes \pi^*\mathcal{O}_{\mathbb{P}^n}(1)|_D) \longrightarrow \cdots
    \end{align*}

    Here, $\pi^*\mathcal{O}_{\mathbb{P}^n}(1)|_D$ is nef and big on $D$ by the characterization of nef and the assumption. 
    Similar to above, there exist a singular Hermitian metric $h^D_{O(1)}$ on $\pi^*\mathcal{O}_{\mathbb{P}^n}(1)|_D$ such that 
    $\mathscr{I}(h^D_{O(1)})=\mathcal{O}_D$ and $i\Theta_{\pi^*\mathcal{O}_{\mathbb{P}^n}(1),h^D_{O(1)}}\geq\varepsilon_D\omega_D$ in the sense of currents for some $\varepsilon_D>0$, where $\omega_D$ is a \kah metric on $D$.
    By the Nadel vanishing theorem, we have $H^q(D,K_D\otimes \pi^*\mathcal{O}_{\mathbb{P}^n}(1)|_D)=0$ for any $1\leq q\leq n-1$.
    
    Hence, we have the following isomorphism 
    \begin{align*}
        0\ne H^{n-1}(X,\Omega_X^{n-1}\otimes \pi^*\mathcal{O}_{\mathbb{P}^n}(1))\cong H^{n-1}(X,\Omega_X^{n-1}(\log D)\otimes \pi^*\mathcal{O}_{\mathbb{P}^n}(1)),
    \end{align*}
    and this is the counterexample.
\end{proof}

\begin{remark}\label{Log V-thm big then with nu condition}
    Corollary \ref{Log V-thm for big in 4}, \ref{Log V-thm of semi-posi + s-Grif} and \ref{Log V-thm of semi-posi + s-Nak} also follow from Theorem \ref{Log V-thm for big no nu condition} and \ref{Log V-thm of semi-posi + s-Grif no nu condition} and $\mathrm{[Ina22,\,Theorem\,1.5]}$, respectively.
\end{remark}

\begin{proof}
    It is sufficient to show that Theorem \ref{Log V-thm for big no nu condition} $\Longrightarrow$ Corollary \ref{Log V-thm for big in 4}.
    By $k$-positivity of $A\otimes\mathcal{O}_X(\Delta)$, there exist smooth Hermitian metrics $h_A$ and $h_j$ on $A$ and $\mathcal{O}_X(D_j)$ respectively,
    such that the curvature form of the induced metric $h_A\prod^s_{j=1}h_j^{a_j}$ on $A\otimes\mathcal{O}_X(\Delta)$, i.e. $i\Theta_{A,h_A}+i\sum^s_{j=1}a_j\Theta_{\mathcal{O}_X(D_j),h_j}$,
    is semipositive. Let $\sigma_j$ be the defining section of $D_j$ then $\sigma_j$ induces the singular Hermitian metric $h_{D_j}=1/|\sigma_j|^2$ on $\mathcal{O}_X(D_j)$ such that $h_{D_j}$ is smooth on $X\setminus D_j$ and pseudo-effective on $X$, i.e. $\frac{i}{2\pi}\Theta_{\mathcal{O}_X(D_j),h_{D_j}}=[D_j]\geq0$ in the sense of currents.

    Define the singular Hermitian metric $h_{A\otimes D}:=h_A\prod^s_{j=1}h_j^{a_j} h_{D_j}^{1-a_j}$ on $A\otimes\mathcal{O}_X(D)$ then $h_{A\otimes D}$ is pseudo-effective.
    We show that $\mathscr{I}(h\otimes h_{A\otimes D})=\mathscr{I}(h)$ on $X$.

    (I) If $x\in X\setminus D$ then we get $\mathscr{I}(h\otimes h_{A\otimes D})_x=\mathscr{I}(h)_x$ by smooth-ness of $h_{A\otimes D}$.

    (II) Fixed a point $x\in D$. By the condition $\nu(-\log h,x)<2\delta\leq2$, we get $\mathscr{I}(h)_x=\mathcal{O}_{X,x}$.
    It is sufficient to show that $\mathcal{O}_{X,x}\subseteq\mathscr{I}(h\otimes h_{A\otimes D})_x$.
    Let $U$ be a open neighborhood of $x$ and $(z_1,\cdots,z_n)$ be a local coordinate chart such that the locus of $D$ is given by $z_1\cdots z_t=0$.
    Since locally the $\sigma_j$ can be taken to be coordinate functions. Here, $f\in\mathscr{I}(h\otimes h_{A\otimes D})(U)=\mathscr{I}(h\otimes\prod^s_{j=1}h_{D_j}^{1-a_j})(U) \iff$
    \begin{align*}
        I_f:=\int_U\frac{|f|^2}{\prod^t_{j=1}|z_j|^{2(1-a_j)}}e^{-\varphi}dV_{\mathbb{C}^n}<+\infty,
    \end{align*}
    where $h=e^{-\varphi}$. 
    By the assumption $\nu(-\log h,x)<2\delta$, i.e. $\nu(\varphi/\delta,x)<2$, we have $\int_Ue^{-\varphi/\delta}dV_{\mathbb{C}^n}<+\infty$.
    From the H\"older's inequality, for any $f\in\mathcal{O}_X(U)$ we get
    \begin{align*}
        I_f\leq\sup_U|f|^2\Bigl(\int_U\prod^t_{j=1}|z_j|^{-2\frac{1-a_j}{1-\delta}}dV_{\mathbb{C}^n}\Bigr)^{1-\delta}\cdot\Bigl(\int_Ue^{-\varphi/\delta}dV_{\mathbb{C}^n}\Bigr)^\delta.
    \end{align*}
    By the assumption $a_j\in(\delta,1]$, the integral of $\prod^t_{j=1}|z_j|^{-2\frac{1-a_j}{1-\delta}}$ is finite where $-\frac{1-a_j}{1-\delta}>-1$, then we get $I_f<+\infty$.
    Hence, $\mathcal{O}_{X,x}\subseteq\mathscr{I}(h\otimes h_{A\otimes D})_x$.

    Therefore, applying Theorem \ref{Log V-thm for big no nu condition} to the big line bundle $A\otimes \mathcal{O}_X(D)\otimes L$ with the singular positive Hermitian metric $h_{A\otimes D}\otimes h$, for any $p,q\geq1$ we obtain 
    \begin{align*}
        &H^q(X,K_X\otimes\mathcal{O}_X(D)\otimes A\otimes L\otimes\mathscr{I}(h))\\
        &\qquad\cong H^q(X,K_X\otimes (A\otimes\mathcal{O}_X(D)\otimes L)\otimes\mathscr{I}(h_{A\otimes D}\otimes h))=0,\\
        &H^n(X,\Omega_X^p(\log D)\otimes A\otimes L\otimes\mathscr{I}(h))\\
        &\qquad\cong H^n(X,\Omega_X^p(\log D)\otimes \mathcal{O}_X(D)^*\otimes (A \otimes\mathcal{O}_X(D)\otimes L)\otimes\mathscr{I}(h_{A\otimes D}\otimes h))=0.
    \end{align*}

    Note that for the case of Corollary \ref{Log V-thm of semi-posi + s-Grif} and \ref{Log V-thm of semi-posi + s-Nak}, we use the fact that if $h$ is ($L^2$-type) strictly (dual) Nakano $\rho_\omega$-positive and $h_{A\otimes D}$ is pseudo-effective then $h\otimes h_{A\otimes D}$ is also ($L^2$-type) strictly (dual) Nakano $\rho_\omega$-positive (see [Wat22b,\,Corollary\,3.19]).
\end{proof}

\begin{remark}
    However, Theorem \ref{Log V-thm of k-posi + psef} and \ref{Log V-thm of k-posi + Grif} and Corollary \ref{Log V-thm of k-posi + Nak} cannot be shown directly from Theorem \ref{Log V-thm of k-posi + psef no nu condition} and \ref{Log V-thm of k-posi + Grif no nu condition} and $\mathrm{[Wat22b,\,Theorem\,6.1]}$, respectively, using the same technique as Remark \ref{Log V-thm big then with nu condition}.
    But, Corollary \ref{Log V-thm of k-posi + Nak} and the cohomology vanishing for $(n,q)$-forms in Theorem \ref{Log V-thm of k-posi + psef} and \ref{Log V-thm of k-posi + Grif} is shown by applying the following vanishing theorem to a singular Hermitian metric created in the same way as Remark \ref{Log V-thm big then with nu condition}.
\end{remark}

\begin{theorem}
    Let $X$ be a compact \kah manifold and $D$ be a simple normal crossing divisor on $X$ and $Y:=X\setminus D$.
    Let $A$ and $L$ be holomorphic line bundles and $E$ be a holomorphic vector bundle equipped with a singular Hermitian metric $h$.
    We assume that there exists a singular Hermitian metric $h_A$ on $A$ such that $h_A|_Y$ is smooth and $k$-positive on $Y$. We have the following
    \begin{itemize}
        \item [($a$)] If $L$ has a singular semi-positive Hermitian metric $h_L$ then for $p,q\geq k$ we get
        \begin{align*}
            H^q(X,K_X\otimes A\otimes L\otimes\mathscr{I}(h_A\otimes h_L))&=0,\\
            H^q(X,\Omega_X^p\otimes A\otimes L\otimes\mathscr{I}(h_A\otimes h_L))&=0.
        \end{align*} 
        \item [($b$)] If $X$ is projective and $h$ is Griffiths semi-positive then for $p,q\geq k$ we get
        \begin{align*}
            H^q(X,K_X\otimes \mathscr{E}(h_A\otimes h))&=0,\\
            H^q(X,\Omega_X^p\otimes \mathscr{E}(h_A\otimes h))&=0.
        \end{align*}
        \item [($c$)] If $X$ is projective and $h$ is $L^2$-type Nakano semi-positive then for $q\geq k$ we get
        \begin{align*}
            H^q(X,K_X\otimes \mathscr{E}(h_A\otimes h))=0.
        \end{align*}
    \end{itemize}
\end{theorem}

Here, two $L^2$-estimates used in the proofs of $(b)$ and $(c)$ corresponding to Theorem \ref{L2-estimate of k-posi + psef} are obtained similar to [Wat22b,\,Theorem\,4.8\,and\,4.9], respectively.
This theorem is shown by using these two $L^2$-estimates and Theorem \ref{L2-estimate of k-posi + psef} and \ref{L2-type Dolbeault isomorphism with sHm} and Lemma \ref{Ext d-equation for hypersurface}.

Finally, the cohomology vanishing for $(n,q)$-forms in Theorem \ref{Log V-thm of k-posi + psef no nu condition}, \ref{Log V-thm for big no nu condition}, \ref{Log V-thm of k-posi + Grif no nu condition} and \ref{Log V-thm of semi-posi + s-Grif no nu condition} follows immediately from already known results $($see \cite{Wat22a},\,\cite{Wat22b}$)$, and this paper gives an another proof using $L^2$-type Dolbeault isomorphisms for logarithmic sheaves.


\vspace{3mm}

{\bf Acknowledgement. } 
I would like to thank my supervisor Professor Shigeharu Takayama for guidance and helpful advice.

$ $

\rightline{\begin{tabular}{c}
    $\it{Yuta~Watanabe}$ \\
    $\it{Graduate~School~of~Mathematical~Sciences}$ \\
    $\it{The~University~of~Tokyo}$ \\
    $3$-$8$-$1$ $\it{Komaba, ~Meguro}$-$\it{ku}$ \\
    $\it{Tokyo, ~Japan}$ \\
    ($E$-$mail$ $address$: watayu@g.ecc.u-tokyo.ac.jp)
\end{tabular}}

\end{document}